\documentclass[preprint,12pt]{elsarticle}

\usepackage{amsmath, amscd, amssymb, mathrsfs, esint}  
\usepackage{amsthm}                                    
\usepackage{color}                                     
\usepackage[colorlinks=true, linkcolor=blue, citecolor=blue, urlcolor=blue]{hyperref}   
\usepackage{graphicx}        

\usepackage[margin=1 in]{geometry} 
\usepackage{enumitem}              
\setlist[enumerate]{
    label=(\roman*),
    leftmargin=*,
    itemsep=1pt,
    topsep=3pt,
}                                  

\newtheorem{thm}{Theorem}[section]     
\newtheorem{lem}[thm]{Lemma}
\newtheorem{prop}[thm]{Proposition}
\newtheorem{coro}[thm]{Corollary}

\newtheorem{rem}[thm]{Remark}

\numberwithin{equation}{section}      

\def\R{\mathbb R}  

\begin{document}

\begin{frontmatter}



\title{Some Reverse Hardy--Littlewood--Sobolev Type Inequalities}


\author[nwpu]{Qianqiao Guo\corref{cor1}}
\ead{gqianqiao@nwpu.edu.cn}

\author[nwpu]{Zhe Pu}
\ead{puzhe@mail.nwpu.edu.cn}

\author[nwpu]{Jiankang Xia}
\ead{jiankangxia@nwpu.edu.cn}

\cortext[cor1]{Corresponding author}

\affiliation[nwpu]{organization={School of Mathematics and Statistics, Northwestern Polytechnical University},
   city={Xi'an},
    postcode={710129},
    state={Shaanxi},
    country={China}
}

\begin{abstract}

\noindent
	
We establish some sharp reverse Hardy--Littlewood--Sobolev (HLS) type inequalities on \(\mathbb{R}^n\) and \(\mathbb{R}_+^n\). Using an operator representation, we overcome the difficulty that the symmetric double-integral structure is unavailable in the half-space setting. 

On \(\mathbb{R}^n\), for \(1 \le n < \alpha\), \(\frac{n}{\alpha} < t < 1\), and \(0 < q < 1\), there holds for nonnegative \(f \) that
\[
\| E_\alpha f \|_{L^{t^\prime}(\mathbb{R}^n)}
\ge \mathscr{C}(n,\alpha,q,t)
\| f \|_{L^1(\mathbb{R}^n)}^{\gamma}
\| f \|_{L^q(\mathbb{R}^n)}^{1-\gamma},
\quad
\gamma := \frac{n - q\alpha - \frac{n}{t^\prime}q}{n(1-q)}
\]
for some $\mathscr{C}(n,\alpha,q,t)>0$
iff \(q > \frac{n}{\alpha}\), where \(E_\alpha\) is the extension operator with Riesz kernel and \(t^\prime\) is the conjugate of \(t\). The sharp constant is achieved when \(\frac{n t^\prime}{n + \alpha t^\prime} \le q < 1\). 

On \(\mathbb{R}_+^n\),  with \(2 \le n < \alpha\), \(\frac{n}{\alpha} < t < 1\), and \(0 < q < 1\), we show for nonnegative \(f \) that
\[
\| \widetilde{E}_\alpha f \|_{L^{t^\prime}(\mathbb{R}_+^n)}
\ge \widetilde{\mathscr{C}}(n,\alpha,q,t)
\| f \|_{L^1(\partial \mathbb{R}_+^n)}^{\widetilde{\gamma}}
\| f \|_{L^q(\partial \mathbb{R}_+^n)}^{1-\widetilde{\gamma}},
\quad
\widetilde{\gamma} := \frac{(n-1) - q(\alpha-1) - \frac{n}{t^\prime}q}{(n-1)(1-q)},
\]
for some $\widetilde{\mathscr{C}}(n,\alpha,q,t)>0$ iff \(q > \frac{n-1}{\alpha-1}\), where \(\widetilde{E}_\alpha\) is the extension operator with Poisson-type kernel. The sharp constant is achieved when \(\frac{t^\prime(n-1)}{n + t^\prime(\alpha-1)} \le q < 1\).

We further extend results to \(q\ge1\).

The proofs use rearrangement inequalities, the sharp Carlson--Levin inequality, and refined pointwise lower bounds for the Riesz and Poisson-type potentials. Our results unify and extend the classical reverse HLS inequalities, especially on \(\mathbb{R}_+^n\).
\end{abstract}




\end{frontmatter}


\section{Introduction}

The classical sharp Hardy--Littlewood--Sobolev (HLS for short) inequality (\cite{1928-HaLi, 1930-HaLi, 1983-Lieb-A, 1963-Sobolev}) states that
\begin{equation}\label{ineq-class-HLS}
\left|\int_{\R^n} \int_{\R^n} f(x) |x-y|^{\alpha-n} g(y) dy dx \right|\le \mathscr{N}(n,\alpha,p)\|f\|_{L^p(\R^n)}\|g\|_{L^t(\R^n)}
\end{equation}
for all $f\in L^p(\R^n), \, g\in L^t(\R^n),$  where $\mathscr{N}(n,\alpha, p)>0, \ 1<p, \, t<\infty$, \ $0<\alpha <n$ and 
$$
\frac{1}{p} + \frac{1}{t}+ \frac{n-\alpha}{n}=2.
$$
Lieb \cite{1983-Lieb-A} proved the existence of extremal functions for the inequality. 
He also classified the extremal functions and computed the best constant in the conformal case $p=t$, as well as in the case $p=2$ or $t=2$, which can be reduced to the conformal case. 
Using the competing symmetry method, Carlen and Loss \cite{1990-CaLo-JF} provided a different proof to obtain the sharp constant and extremal functions in the conformal case.
Later, Frank and Lieb \cite{2010-FrLi-CV} gave another proof based on the reflection positivity of inversions in spheres.
In \cite{2012-FrLi-B}, Frank and Lieb further developed a rearrangement-free approach introduced in \cite{2012-FrLi-A} to derive the sharp constant in \eqref{ineq-class-HLS}.

The HLS inequality \eqref{ineq-class-HLS} plays a fundamental role in geometry and analysis. 
For instance, the sharp HLS inequality implies the Moser--Trudinger--Onofri inequality and Beckner inequalities \cite{1993-Beckner-A}, which are closely related to the prescribing curvature problems. 
It also implies Sobolev inequality and Gross's logarithmic Sobolev inequality \cite{1976-Gross}.

Over the past decades, the sharp HLS inequality has been extended in several directions.
We briefly review several developments that are most closely related to the present work.

One direction concerns the extension of the classical exponent range to $\alpha>n$, leading to the so-called reverse HLS inequality. This inequality was established independently by Beckner \cite{2015-Beckner}, Dou and Zhu \cite{2015-DouZhu-IM-R}.
More precisely, let $1 \le n < \alpha$ and $0<p,t<1$ satisfy $\frac{1}{p} + \frac{1}{t} + \frac{n-\alpha}{n} = 2$. Then
\begin{equation}\label{ineq-Re-HLS}
\int_{\R^n} \int_{\R^n}f(x) |x-y|^{\alpha-n}g(y) dy dx 
\ge 
\mathscr{N}(n,\alpha, p)\|f\|_{L^p(\R^n)}\|g\|_{L^t(\R^n)}
\end{equation}
for all nonnegative functions $f\in L^p(\R^n), \, g\in L^t(\R^n)$, where $\mathscr{N}(n,\alpha, p)>0$. 
Here by convention, for $p\neq 0$, we write $f\in L^p(\mathbb{R}^n)$ if $\int_{\mathbb{R}^n}|f(x)|^p\,dx<\infty$, and define $\|f\|_{L^p(\mathbb{R}^n)} := \left(\int_{\mathbb{R}^n}|f(x)|^p\,dx\right)^{1/p}$, which is only a quasi-norm when $0<p<1$, and fails even to be positive definite when $p<0$. 
The above inequality \eqref{ineq-Re-HLS} is also equivalent to 
\begin{equation} \label{ineq-Re-HLS-operator}
\| E_\alpha f \|_{L^{t^\prime}(\mathbb{R}^n)} 
\ge \mathscr{N}(n,\alpha,p)\|f\|_{L^p(\mathbb{R}^n)},
\end{equation}
where $E_\alpha$ denotes the extension operator with Riesz kernel
\[
E_\alpha f(x) := \int_{\mathbb{R}^n} |x-y|^{\alpha-n}f(y) dy, \ \ x \in \mathbb{R}^n.
\]
Throughout this paper, $t$ and $t^\prime$ denote conjugate exponents satisfying $1/t+1/t^\prime=1$. 

Another line of research concerns extending the sharp inequality to more general geometric settings.
For example, HLS-type inequalities have been established on the Heisenberg group by Frank and Lieb \cite{2012-FrLi-A}, on compact Riemannian manifolds by Han and Zhu \cite{2016-Han-Zhu-JDE}, and on the upper half-space $\mathbb{R}^n_+:=\{ x=(x_1,\dots,x_n)\in\mathbb{R}^n : x_n>0 \}$ by Dou and Zhu \cite{2015-DouZhu-IM-H}.  
Moreover, the reverse HLS inequality on the upper half-space  was indicated in \cite{2015-DouZhu-IM-R, 2015-DouZhu-IM-H} and also proved by Ngô and Nguyen \cite{2017-NgoNgu-RH} via a different method, which states that
\begin{equation}\label{ineq-Re-HLS-operator-half-space} 
\| \widetilde E_\alpha f \|_{L^{t^\prime}(\mathbb{R}^n_+)} \ge \widetilde{\mathscr{N}}(n,\alpha,p) \| f \|_{L^p(\partial \mathbb{R}^n_+)},
\end{equation}
for any nonnegative $f \in L^p(\partial \mathbb{R}^n_+)$, where $\widetilde{\mathscr{N}}(n,\alpha,p)>0$, $2\le n<\alpha,$ $0<p,t<1$ satisfying
\[
\frac{n-1}{n} \cdot \frac{1}{p} + \frac{1}{t} + \frac{n-\alpha+1}{n} = 2,
\]
 $\widetilde E_\alpha$ denotes the extension operator with Poisson-type kernel
\[
\widetilde{E}_\alpha f(x) := \int_{\partial \mathbb{R}^n_+} |x-y|^{\alpha-n}f(y) \, dy, \ \  x = (x', x_n) \in \mathbb{R}^n_+.
\]

The double weighted HLS inequality, known as the Stein--Weiss inequality, was established by Stein and Weiss \cite{1958-Stein-Weiss} on $\mathbb R^n$.  
The existence of extremal functions for the Stein--Weiss inequality was obtained by Lieb \cite{1983-Lieb-A}. 
We also refer to \cite{2008-Chen-Li-PAMS, 2008-Beckner, 2012-Han-Lu-Zhu-NA, 2016-DouJB-CCM, 2019-Chen-Lu-Tao-ANS, 2019-Chen-Lu-Tao-JFA} for further developments. 
Another type of extension replaces the Riesz kernel by the Poisson-type kernel. 
On the upper half-space $\mathbb R^n_+$, one can consider HLS-type inequalities involving Poisson-type kernel 
$$
K_{\alpha,\beta}(x',x_{n})
:=\frac{x_{n}^{\beta}}{(|x'|^{2}+x_{n}^{2})^{\frac{n-\alpha}{2}}},
\quad x=(x',x_{n})\in\mathbb{R}^n_+,
$$
for which we refer to
\cite{2008-HangFB-WangXD-YanXD-CPAM, 2014-ChenSB-IMRN, 2017-DouJB-GuoQQ-ZhuMJ-AD, 2019-Chen-Lu-Tao-AMSES, 2020-Gluck-JFA, 2023-LiuZ, 2023-Dai-Hu-Liu}
and the references therein.

Recently, a new family of reverse HLS inequalities was obtained by Carrillo et al. in \cite{2019-Carrillo-Delgadino-Dolbeault-Frank-Hoffmann-JMPA}, motivated by the analysis of nonnegative solutions to certain Keller--Segel type equations.
That is, for $1\le n<\alpha$, $0<q<1$ and $\overline{\gamma}:=\frac{2n-q(n+\alpha)}{n(1-q)}$, the inequality
\begin{equation}\label{ineq-CDDFH}  
\int_{\mathbb R^n}\int_{\mathbb R^n}f(x)|x-y|^{\alpha-n}f(y)dxdy
\ge \mathscr{F}(n,\alpha, q)\|f\|^{\overline{\gamma}}_{L^1(\mathbb R^n)}\|f\|^{2-\overline{\gamma}}_{L^q(\mathbb R^n)}
\end{equation}
holds for all nonnegative $f\in L^1 \cap L^q(\R^n)$. 
Obviously, by using the reverse HLS inequality \eqref{ineq-Re-HLS} (for the conformal case) and H\"{o}lder's inequality, the inequality \eqref{ineq-CDDFH} holds for $\frac{2n}{n+\alpha}<q<1$. 
It is interesting that Carrillo et al. \cite{2019-Carrillo-Delgadino-Dolbeault-Frank-Hoffmann-JMPA} show that the inequality holds for $\mathscr{F}(n,\alpha, q)>0$ if and only if $\frac{n}{\alpha}<q<1$; and if either $n=1, 2$ or if $n\ge 3$ and $q\ge \min \{1-\frac{2}{n}, \frac{2n}{n+\alpha} \}$, then there is a radial positive, non-increasing, bounded function which achieves the equality case. 
Similar results for the classical HLS inequality can also be seen in \cite{2017-Calvez-Carrillo-Hoffmann} and \cite{2025-ZhangXQ-JMAA}.

Motivated by \cite{2019-Carrillo-Delgadino-Dolbeault-Frank-Hoffmann-JMPA}, it is quite natural to establish analogues of \eqref{ineq-CDDFH} on the upper half-space $\mathbb{R}^n_+$.
However, a main difficulty arises from the fact that the symmetric double-integral structure present in \cite{2019-Carrillo-Delgadino-Dolbeault-Frank-Hoffmann-JMPA} is no longer available in the half-space setting. 
This observation leads us to consider the reverse HLS inequality in its operator form.

We first establish inequalities on $\mathbb{R}^n$, which are analogues of \eqref{ineq-CDDFH} based on the reverse HLS inequality \eqref{ineq-Re-HLS-operator}, to highlight the main ideas without being obscured by additional technical difficulties. 
Then we consider the cases on the upper half-space $\mathbb{R}^n_+$, analogues of \eqref{ineq-CDDFH} based on the reverse HLS inequality \eqref{ineq-Re-HLS-operator-half-space}.

Heuristically, by \eqref{ineq-Re-HLS-operator} and H\"{o}lder's inequality, for $\frac{n}{\alpha}<p\le q<1$ we have
\begin{align}\label{relation-general-t}
\| E_\alpha f \|_{L^{t^\prime}(\mathbb{R}^n)} 
&\ge \mathscr{N}(n,\alpha,p) \|f\|_{L^p(\mathbb{R}^n)} \nonumber\\
&\ge \mathscr{N}(n,\alpha,p) 
\|f\|^{\gamma}_{L^1(\mathbb R^n)}
\|f\|^{1-\gamma}_{L^q(\mathbb R^n)},
\end{align}
where
\begin{equation} \label{def-gamma}
\gamma :=\frac{n-q\alpha-\frac{n}{t^\prime}q}{n(1-q)}
\end{equation}
is determined by scaling and homogeneity. 
It is therefore natural to ask whether the inequality \eqref{relation-general-t}  remains valid in a larger range of exponents.

The first main result gives an affirmative answer to this question.

\begin{thm}\label{them-Rn}    
Assume that $1\le n<\alpha$, $\frac{n}{\alpha}<t<1$, $0<q<1$ and $\gamma$ is defined as above. 
Then for any nonnegative $f\in L^1(\R^n)\cap L^q(\R^n)$, there holds
\begin{equation}\label{ineq-main-Rn}   
\|E_\alpha f\|_{L^{t^\prime}(\mathbb{R}^n)}
\ge
\mathscr{C}(n,\alpha,q,t)
\|f\|_{L^1(\mathbb{R}^n)}^\gamma
\|f\|_{L^q(\mathbb{R}^n)}^{1-\gamma},
\end{equation}
for some positive constant $\mathscr{C}(n,\alpha,q,t)$ if and only if $\frac{n}{\alpha}<q<1$. 
Moreover, in the range $\frac{nt^\prime}{n+\alpha t^\prime}\le q<1$, the sharp constant $\mathscr{C}(n,\alpha,q,t)$ is achieved by a nontrivial, radially symmetric, non-increasing function which is positive almost everywhere.
\end{thm}

\begin{rem}
The assumption $\frac{n}{\alpha}<t<1$ is necessary. 
Indeed, suppose that $0 < t \le \frac{n}{\alpha}$. 
Let $f \in C_c^\infty(\mathbb{R}^n)$ be nonnegative. 
Then
\[
E_\alpha f(x)
=
\int_{\mathbb{R}^n}
|x-y|^{\alpha-n}f(y)\,dy
\sim
|x|^{\alpha-n}
\int_{\mathbb{R}^n} f(y)\,dy
\]
as $|x| \to \infty$.
Since $ t^\prime (n-\alpha)\le n$, we have 
\[
\int_{\mathbb{R}^n}
\bigl(E_\alpha f(x)\bigr)^{t^\prime}\,dx
=
\infty.
\]
Consequently, $\|E_\alpha f\|_{L^{t^\prime}(\R^n)}=0$, and therefore no meaningful reverse estimate of the form
\eqref{ineq-main-Rn} can hold in this regime. 
\end{rem}

\begin{rem} 
\begin{enumerate}
\item 
Let $q=p=\frac{nt^\prime}{n+\alpha t^\prime}$, that is, \(\gamma=0\). Then \eqref{ineq-main-Rn} reduces to the reverse HLS inequality \eqref{ineq-Re-HLS-operator}.


\item The inequality \eqref{ineq-main-Rn} also implies the reverse HLS-type inequality  \eqref{ineq-CDDFH} obtained by Carrillo  et al.~\cite{2019-Carrillo-Delgadino-Dolbeault-Frank-Hoffmann-JMPA}.
Indeed, for any nonnegative $f\in L^1(\R^n)\cap L^q(\R^n)$ with $q\in \left( \frac{n}{\alpha},1 \right)$, choosing $t=q$ and applying H\"older's inequality together with \eqref{ineq-main-Rn}, we obtain
\begin{align*}
\int_{\mathbb R^n}\int_{\mathbb R^n}f(x)|x-y|^{\alpha-n}f(y)\,dxdy
&\ge 
\|E_\alpha f\|_{L^{q^\prime}(\R^n)}
\|f\|_{L^{q}(\R^n)} \\
&\ge 
\mathscr{C}(n,\alpha,q,q)
\|f\|_{L^{1}(\R^n)}^\gamma
\|f\|_{L^{q}(\R^n)}^{1-\gamma}
\|f\|_{L^{q}(\R^n)} \\
&=
\mathscr{C}(n,\alpha,q,q)
\|f\|_{L^{1}(\R^n)}^{\overline{\gamma}}
\|f\|_{L^{q}(\R^n)}^{2-\overline{\gamma}}.
\end{align*}
\end{enumerate}
\end{rem}

Theorem \ref{them-Rn} establishes a necessary and sufficient condition for the validity of the inequality \eqref{ineq-main-Rn}, namely, $\frac{n}{\alpha} < q < 1$.
Regarding the existence of optimizers, the theorem proves the sufficient condition $\frac{n t^\prime}{n + \alpha t^\prime} \le q < 1$, while the necessity remains open.

The optimizers of \eqref{ineq-main-Rn} have been explicitly characterized in the conformally invariant case $q=t=\frac{2n}{n+\alpha}$ in \cite{2015-Beckner, 2015-DouZhu-IM-R, 2017-NgoNgu-R} and are given by translations, dilations and constant multiples of 
$$
f(x)=\left(1+|x|^{2}\right)^{-\frac{n+\alpha}{2}}.
$$
Consequently, the optimal constant in \eqref{ineq-main-Rn} is
$$
{\mathscr{C}}\left(n,\alpha,\frac{2n}{n+\alpha},\frac{2n}{n+\alpha}\right)
=
{\mathscr{N}}\left( n,\alpha,\frac{2n}{n+\alpha} \right)
=
\pi^{\frac{n-\alpha}{2}}
\frac{\Gamma(\frac{\alpha}{2})}{\Gamma(\frac{n+\alpha}{2})}
\left(\frac{\Gamma(n)}{\Gamma(\frac{n}{2})}\right)^{\frac{\alpha}{n}}.
$$
As it will be shown in Section \ref{Sec-Rn}, one can also characterize the optimizers in the special case $\alpha=n+2$.
If $\frac{n}{n+2}<q,t<1$, then the optimizers for (\ref{ineq-main-Rn}) are given by translations, dilations and constant multiples of 
$$
f(x)=\left(1+|x|^{2}\right)^{-\frac{1}{1-q}}.
$$
And the sharp constant
\[
\scalebox{0.91}{$\displaystyle
\mathscr{C}(n,n+2,q,t)
=
\frac{1}{\pi}
\left(
\frac{n(1-q)}{q(n+2)-n}
\left(\frac{q(n+2)-n}{2q}\right)^{\frac{2}{n(1-q)}}
\left(\frac{\Gamma\!\left(\frac{1}{1-q}\right)}{\Gamma\!\left(\frac{1}{1-q}-\frac{n}{2}\right)}\right)^{\frac{2}{n}}
\right)^{1+\frac{n}{2t^\prime}}
\left(
\frac{\Gamma\!\left(-t^\prime-\frac{n}{2}\right)}{\Gamma(-t^\prime)}
\right)^{\frac{1}{t^\prime}}
$}.
\]

It is natural to ask whether an analogue of \eqref{ineq-main-Rn} continues to hold in the case \(q\ge 1\). 
The following theorem gives an affirmative answer for \(q>1\).

\begin{thm}\label{them-Rn-q-bigger-1}     
Assume that $1\le n<\alpha$ and $\frac{n}{\alpha}<t<1<q$. Let $\gamma>1$ be defined as \eqref{def-gamma}.
Then there holds
\begin{equation}\label{ineq-main-Rn-q-bigger-1}  
	\|E_\alpha f\|_{L^{t^\prime}(\R^n)} \|f\|^{\gamma-1}_{L^q(\R^n)} \geq \mathscr{C}(n,\alpha,q,t)\|f\|^\gamma_{L^1( \R^n)}
\end{equation}
for any nonnegative function $ f\in L^1(\R^n) \cap L^q(\R^n)$, for some positive constant $\mathscr{C}(n,\alpha,q,t)$. 
Moreover, there exists a radially symmetric, non-increasing, bounded function with compact support such that the sharp constant  $\mathscr{C}(n,\alpha,q,t)$ is achieved. 
\end{thm}

At the endpoint \(q=1\), the inequality takes the following logarithmic form.
\begin{thm}\label{them-Rn-q=1-log}    
Assume that $1\le n<\alpha$ and $\frac{n}{\alpha}<t<1$. Then there holds
\begin{equation}\label{ineq-main-Rn-log}
\int_{\R^n} f \log f \, dx
+
\frac{n}{\alpha-\frac{n}{t}}
\log\left(
\frac{\|E_\alpha f\|_{L^{t^\prime}(\R^n)}}{\mathscr C(n,\alpha,1,t)}
\right)
\ge 0,
\end{equation}
for any nonnegative function $f \in L^1(\R^n)$ such that $\int_{\R^n} f(x)\,dx = 1$ and $f \log f \in L^1(\R^n)$, for some positive constant $\mathscr{C}(n,\alpha,1,t)$.
Moreover, there exists a radially symmetric, non-increasing, bounded function such that the equality holds.
\end{thm}

Having established the corresponding inequalities on the whole space \(\mathbb R^n\), we now turn to the upper half-space \(\mathbb R^n_+\). 
The following theorems provide the analogue of Theorem~\ref{them-Rn} in this setting.

\begin{thm}\label{them-half-space}   
Let $2\le n<\alpha $, $\frac{n}{\alpha}<t<1$ and $q\in (0,1)$. 
Define 
\begin{equation}\label{def-gamma-half-space}
\widetilde\gamma :=\frac{(n-1)-q(\alpha-1)-\frac{n}{t^\prime}q}{(n-1)(1-q)}. 
\end{equation}
Then for any nonnegative $f\in L^1(\partial \R^n_+) \cap L^q(\partial \R^n_+)$, there holds 
\begin{equation}\label{Ineq-Main-half-space} 
	\|\widetilde E_\alpha f\|_{L^{t^\prime}(\R^n_+)}\geq \widetilde{\mathscr{C}}(n,\alpha,q,t)\|f\|^{\widetilde\gamma}_{L^1(\partial \R^n_+)}\|f\|^{1-\widetilde\gamma}_{L^q(\partial \R^n_+)}
\end{equation}
 for some positive constant $\widetilde{\mathscr{C}}(n,\alpha,q,t)$ if and only if $\frac{n-1}{\alpha-1}<q<1$.
Moreover, in the range $\frac{t^\prime(n-1)}{n + t^\prime(\alpha-1)} \le q < 1$, the sharp constant $\widetilde{\mathscr{C}}(n,\alpha,q,t)$ is achieved by a nontrivial, radially symmetric, non-increasing function which is positive almost everywhere. 
\end{thm}

\begin{rem} Let $q=\frac{t^\prime(n-1)}{n+t^\prime(\alpha -1)}$, that is, $\widetilde\gamma=0$. 
Then the inequality \eqref{Ineq-Main-half-space} reduces to the reverse HLS inequality \eqref{ineq-Re-HLS-operator-half-space} on the upper half-space $\R^n_+$. 
\end{rem}

As in the whole-space setting, the sharp constant can be computed explicitly in several important cases.
The optimizers of \eqref{Ineq-Main-half-space} have been explicitly characterized in the conformally invariant case $q=\frac{2(n-1)}{n+\alpha-2}$ and $t=\frac{2n}{n+\alpha}$ in \cite{2017-NgoNgu-RH} and are given by translations, dilations and constant multiples of 
$$
f(x)=\left(1+|x|^{2}\right)^{-\frac{n+\alpha}{2}}, \quad x\in \partial \mathbb{R}^n_+.
$$
For the special conformal case $\alpha=n+2$, $q=\frac{n-1}{n}$ and $t=\frac{n}{n+1}$, the sharp constant in \eqref{Ineq-Main-half-space} was computed by \cite{2017-NgoNgu-RH} and is given by
\[
\widetilde{\mathscr{C}}\left(n,n+2,\frac{n-1}{n},\frac{n}{n+1}\right) 
=
\widetilde{\mathscr N} \left(n,n+2,\frac{n-1}{n}\right)
=
\frac{2^{\frac{1}{n}-1}}{\pi}
\left( \frac{\Gamma(n)}{\Gamma(\frac{n}{2})} \right)^\frac{1}{n}
\left( \frac{\Gamma(n-1)}{\Gamma(\frac{n-1}{2})} \right)^\frac{1}{n-1}.
\]
Our result goes beyond the conformally invariant setting. 
In Section~\ref{Sec-half-space}, we characterize the optimizers for \eqref{Ineq-Main-half-space} in the case $\alpha=n+2$, $\frac{n}{n+2}<t<1$ and $\frac{n-1}{n+1}<q<1$, which are given, up to translations, dilations, and constant multiples, by
\[
f(x)=\left(1+|x|^{2}\right)^{-\frac{1}{1-q}},\quad x\in \partial \mathbb{R}^n_+.
\]
The corresponding sharp constant $\widetilde{\mathscr{C}}(n,n+2,q,t)$ is given by \eqref{const-sharp-half-space}.


We next consider the case \(q>1\).


\begin{thm}\label{them-half-space-q-bigger-1}     
Assume that $2\le n<\alpha$ and $\frac{n}{\alpha}<t<1<q$.
Let $\widetilde\gamma>1$ be defined as \eqref{def-gamma-half-space}.
Then there holds
\begin{equation}\label{ineq-main-half-space-q-bigger-1}  
	\|\widetilde E_\alpha f\|_{L^{t^\prime}(\R^n_+)} \|f\|^{\widetilde{\gamma}-1}_{L^q(\partial\R^n_+)} \geq \widetilde{\mathscr{C}}(n,\alpha,q,t)\|f\|^{\widetilde\gamma}_{L^1( \partial\R^n_+)}
\end{equation}
for any nonnegative function $ f\in L^1(\partial\R^n_+) \cap L^q(\partial\R^n_+)$, for some positive constant $\widetilde{\mathscr{C}}(n,\alpha,q,t)$. 
Moreover, there exists a radially symmetric, non-increasing, bounded function with compact support such that the sharp constant  $\widetilde{\mathscr{C}}(n,\alpha,q,t)$ is achieved. 
\end{thm}

As in the whole-space setting, The endpoint case \(q=1\) leads to the following logarithmic inequality.

\begin{thm}\label{them-half-space-q=1-log}    
Assume that $2\le n<\alpha$ and $\frac{n}{\alpha}<t<1$. Then there holds
\begin{equation}\label{ineq-main-half-space-log}
\int_{\partial\R^n_+} f \log f \, dx
+
\frac{n-1}{\alpha-\frac{n}{t}}
\log\left(
\frac{\|\widetilde E_\alpha f\|_{L^{t^\prime}(\R^n_+)}}{\widetilde{\mathscr C}(n,\alpha,1,t)}
\right)
\ge 0,
\end{equation}
for any nonnegative function $f \in L^1(\partial\R^n_+)$ such that $\int_{\partial\R^n_+} f(x)\,dx = 1$ and $f \log f \in L^1(\partial\R^n_+)$, for some positive constant $\widetilde{\mathscr{C}}(n,\alpha,1,t)$.
Moreover, there exists a radially symmetric, non-increasing, bounded function such that the equality holds.
\end{thm}

This paper is organized as follows.
Section~\ref{Sec-Rn} is devoted to the proofs of the reverse HLS inequalities on $\mathbb{R}^n$, 
whereas Section~\ref{Sec-half-space} is devoted to the proofs of the corresponding inequalities on the upper half-space $\mathbb{R}^n_+$.

\section{Proofs of theorems on \texorpdfstring{$\mathbb{R}^n$}{Rn}} \label{Sec-Rn}

For $0<q<1$, let $f \in L^{1}(\mathbb{R}^{n}) \cap L^{q}(\mathbb{R}^{n})$ be a function vanishing at infinity. 
Denote by $f^{*}$ its symmetric decreasing rearrangement; see \cite[Chapter~3]{2001-Lieb-Loss-B} for details.

We begin with the following lemma collecting some basic properties of rearrangements.

\begin{lem}\label{lemma-rearrangement} 
Assume $1\le n<\alpha$.
\begin{enumerate}
\item If $f\in L^p(\R^n)$ for $p>0$, then $f^*\in L^p(\R^n)$ and $\|f\|_{L^p(\R^n)}=\|f^*\|_{L^p(\R^n)}$.


\item The function $E_\alpha f^*$ is radially symmetric. Moreover, if $f\not\equiv 0$, then $E_\alpha f^*$ is strictly increasing.

\item For any nonnegative function $f$, 
\begin{equation}\label{ineq-rearrangement-Rn} 
\|E_\alpha f^*\|_{L^{t^\prime}(\R^n)} \le \|E_\alpha f\|_{L^{t^\prime}(\R^n)},
\end{equation}
with equality if and only if $f(x)=f^*(x+x_0)$ for some $x_0\in\R^n$.
\end{enumerate} 
\end{lem}

\begin{proof}
$(i)$ is standard. The proofs of $(ii)-(iii)$ are similar to that in \cite[Proposition~9]{1976-Brascamp-Lieb} and \cite[Lemma~1]{2017-NgoNgu-R}. We omit them here.

\end{proof}

To establish the reverse HLS inequality, we use the following Carlson--Levin inequality.
The inequality goes back to F.~Carlson \cite{1934-Carlson}, and its sharp form was obtained by V.~Levin \cite{1948-Levin}.
The statement below is taken from \cite[Lemma~1.2]{2019-Nguyen-AAM}; see \cite[Lemma~4.1]{2004-Lutwak-Yang-Zhang} for a proof.
For later reference, we state all relevant cases. In the present section we mainly use the range $\frac{n}{n+\lambda}<q<1$, whereas the cases $q\ge1$ will be used in the next section \ref{sec-q>1}.


\begin{lem}[Carlson--Levin inequality \cite{1934-Carlson,1948-Levin,2004-Lutwak-Yang-Zhang,2019-Nguyen-AAM}]  \label{lemma-CLI}  
Let $\lambda>0$ and $q>\frac{n}{n+\lambda}$. There is a constant $\mathscr L(n,\lambda,q)>0$ such that for all $f\ge0$, one has
\begin{enumerate}

\item if $n/(n+\lambda)<q<1$, then
\begin{equation}\label{CLI}
\int_{\mathbb R^n} f(x)|x|^{\lambda}\,dx
\ge
\mathscr L(n,\lambda,q)
\,
\|f\|_{L^q(\mathbb R^n)}^{\frac{q\lambda}{n(1-q)}}
\,
\|f\|_{L^1(\mathbb R^n)}^{\,1-\frac{q\lambda}{n(1-q)}} 
\end{equation}
with
\[
\mathscr L(n,\lambda,q)
=
\frac{n(1-q)}{(n+\lambda)q-n}
\left(
\frac{(n+\lambda)q-n}
     {\lambda^q q}
\right)^{\frac{\lambda}{n(1-q)}}
\left(
\frac{
\Gamma\!\left(\frac n2\right)
\Gamma\!\left(\frac1{1-q}\right)}
{
2\pi^{n/2}
\Gamma\!\left(\frac1{1-q}-\frac n\lambda\right)
\Gamma\!\left(\frac n\lambda\right)}
\right)^{\frac{\lambda}{n}} ,
\]
and equality holds if and only if
\begin{equation}
f(x)=(1+|x|^\lambda)^{-\frac1{1-q}}
\end{equation}
up to dilations and constant multiples;

\item if $q=1$, then
\begin{equation}\label{LCLI}
\int_{\mathbb R^n} |x|^\lambda f(x)\,dx
\ge
\mathscr L(n,\lambda,1)
\exp\left(
-\frac{\lambda}{n}
\frac{\int_{\mathbb R^n}f(x)\log f(x)\,dx}
{\int_{\mathbb R^n}f(x)\,dx}
\right)
\left(\int_{\mathbb R^n} f(x)\,dx\right)^{\frac{n+\lambda}{n}} 
\end{equation}
with
\[
\mathscr L(n,\lambda,1)
=
\frac{n}{\lambda e}
\left(
\frac{
n
\Gamma\!\left(\frac n2\right)}
{
2\pi^{n/2}
\Gamma\!\left(\frac n\lambda+1\right)}
\right)^{\frac{\lambda}{n}},
\]
and equality holds if and only if
\[
f(x)=e^{-|x|^\lambda}
\]
up to dilations and constant multiples;

\item if $q>1$, then
\begin{equation}\label{CLI+}
\int_{\mathbb R^n} f(x)|x|^\lambda\,dx
\ge
\mathscr L(n,\lambda,q)
\,
\|f\|_{L^1(\mathbb R^n)}^{\frac{\lambda q+n(q-1)}{n(q-1)}}
\,
\|f\|_{L^q(\mathbb R^n)}^{-\frac{\lambda q}{n(q-1)}}
\end{equation}
with
\[
\mathscr L(n,\lambda,q)
=
\frac{n(q-1)}{\lambda q+n(q-1)}
\left(
\frac{\lambda^q q}
     {\lambda q+n(q-1)}
\right)^{\frac{\lambda}{n(q-1)}}
\left(
\frac{
\Gamma\!\left(\frac n2\right)
\Gamma\!\left(\frac q{q-1}+\frac n\lambda\right)}
{
2\pi^{n/2}
\Gamma\!\left(\frac q{q-1}\right)
\Gamma\!\left(\frac n\lambda\right)}
\right)^{\frac{\lambda}{n}}, 
\]
and equality holds if and only if
\[
f(x)=(1-|x|^\lambda)_+^{\frac1{q-1}}
\]
up to dilations and constant multiples.

\end{enumerate}
\end{lem}

\subsection{Rough reverse HLS inequality on \texorpdfstring{$\mathbb{R}^n$}{Rn}}
Define 
\[
F[f]
:=\frac{\|E_\alpha f\|_{L^{t^\prime}(\R^n)}}{\|f\|^\gamma_{L^1(\R^n)}\|f\|^{1-\gamma}_{L^q(\R^n)}}.
\]
Consider the following minimization problem:
\begin{equation}\label{mini-problem-Rn} 
\mathscr{C}(n,\alpha,q,t)
:=
\inf\left\{
    F[f]:0\le f\in L^1 (\R^n)\cap L^q(\R^n), f \not\equiv 0 
    \right\}.
\end{equation}

\medskip
We first give a necessary and sufficient condition for the inequality ($\ref{ineq-main-Rn}$).

\begin{prop}\label{Prop-SNC-Rn}  
Let $1 \le n < \alpha$ and $\frac{n}{\alpha}<t<1$. 
\begin{enumerate}
\item If  $0 < q \le \frac{n}{\alpha}$, then $\mathscr{C}(n,\alpha,q,t)=0$.

\item If $\frac{n}{\alpha} < q < 1$, then $\mathscr{C}(n,\alpha,q,t) > 0$.
\end{enumerate}
\end{prop}

\begin{proof}
We first prove $(i)$, which is divided into following two cases.

\textbf{Case~1.}
Assume that $0 < q < \frac{n}{\alpha} < t < 1$, so that $\gamma<1$. 

In this case, we adapt the argument of \cite{2019-Carrillo-Delgadino-Dolbeault-Frank-Hoffmann-JMPA}. 
Define 
\[
I_\alpha[f]
:=
\int_{\mathbb{R}^n} f(x)\,E_\alpha f(x)\,dx
=
\int_{\mathbb{R}^n}\int_{\mathbb{R}^n}
f(x)\,|x-y|^{\alpha-n}\,f(y)\,dy\,dx .
\]
By H\"older's inequality, for every nonnegative
\(0\not\equiv f\in L^1\cap L^q(\mathbb{R}^n)\), we have
\[
\mathscr{C}(n,\alpha,q,t)
\le
F[f]
=
\frac{\|E_\alpha f\|_{t^\prime}}{\|f\|_1^{\gamma}\,\|f\|_q^{1-\gamma}}
\le
\frac{I_\alpha[f]}{\|f\|_1^{\gamma}\,\|f\|_q^{1-\gamma}\,\|f\|_t}
=: G[f].
\]
Therefore, it suffices to construct a family \(\{f_\varepsilon\}\) such that
\[
\lim_{\varepsilon \to 0} G[f_\varepsilon] = 0.
\]

Let $f, g \in C_c^{\infty}(\mathbb{R}^n)$ be nonnegative functions and
\[
\int_{\mathbb{R}^n} g(x)\,dx = 1.
\]
For \(\varepsilon>0\) sufficiently small, define
\[
f_\varepsilon(x) := f(x) + \varepsilon^{-\alpha}\,g\!\left(\frac{x}{\varepsilon}\right).
\]
A direct computation yields
\begin{equation}\label{NE-1}  
\|f_\varepsilon\|_{L^1(\R^n)}
=
\|f\|_{L^1(\R^n)} + \varepsilon^{\,n-\alpha}.
\end{equation}
Moreover, using the elementary inequality
\begin{equation}\label{elementary inequality}
\min\{1,2^{1-a}\}(x+y)^a\le x^a+y^a
\le
\max\{1,2^{1-a}\}(x+y)^a,
\ \  a,x,y\ge0,
\end{equation}
we obtain
\begin{equation}\label{NE-2}
\|f_\varepsilon\|_{L^q(\R^n)}^q
\ge
C_1
\bigl(
\|f\|_{L^q(\R^n)}^q+\varepsilon^{\,n-q\alpha}
\bigr),
\end{equation}
and similarly, 
\begin{equation}\label{NE-3}
\|f_\varepsilon\|_{L^t(\R^n)}^t
\ge
C_2
\bigl(
\|f\|_{L^t(\R^n)}^t+\varepsilon^{\,n-t\alpha}
\bigr).
\end{equation}
Finally, expanding \(I_\alpha[f_\varepsilon]\), we get 
\begin{align}
I_\alpha[f_\varepsilon]
&=
I_\alpha[f]
+
2\varepsilon^{-\alpha}
\int_{\mathbb{R}^n}
\int_{\mathbb{R}^n}
f(x)|x-y|^{\alpha-n}
\sigma\!\left(\frac{y}{\varepsilon}\right)
\,dx\,dy
\nonumber\\
&\quad
+
\varepsilon^{-2\alpha}
\int_{\mathbb{R}^n}
\int_{\mathbb{R}^n}
\sigma\!\left(\frac{x}{\varepsilon}\right)
|x-y|^{\alpha-n}
\sigma\!\left(\frac{y}{\varepsilon}\right)
\,dx\,dy
\nonumber\\
&=
I_\alpha[f]
+
2\left(
\int_{\mathbb{R}^n}
f(x)|x|^{\alpha-n}\,dx
+o_\varepsilon(1)
\right)
\varepsilon^{\,n-\alpha}
+
C_3\varepsilon^{\,n-\alpha}.   \label{NE-4}
\end{align}
Here \(o_\varepsilon(1)\to0\) as \(\varepsilon\to0\).

Combining \eqref{NE-1}--\eqref{NE-4}, and using 
\[
q\alpha<n<t\alpha<\alpha,
\]
we deduce that
\begin{align*}
G[f_\varepsilon]
&\le
\frac{
I_\alpha[f]+C_4\varepsilon^{\,n-\alpha}
}{
\left(
\|f\|_{L^1(\R^n)}+\varepsilon^{\,n-\alpha}
\right)^\gamma
\left[
C_1
\left(
\|f\|_{L^q(\R^n)}^q+\varepsilon^{\,n-q\alpha}
\right)
\right]^{\frac{1-\gamma}{q}}
\left[
C_2
\left(
\|f\|_{L^t(\R^n)}^t+\varepsilon^{\,n-t\alpha}
\right)
\right]^{\frac1t}
}
\\
&=
\frac{
C_4\varepsilon^{\,n-\alpha}
}{
\left(C_1\|f\|_{L^q(\R^n)}^q\right)^{\frac{1-\gamma}{q}}
C_2^{\,1/t}
\,
\varepsilon^{\,\gamma(n-\alpha)+\frac{n}{t}-\alpha}
}
\bigl(1+o_\varepsilon(1)\bigr).
\end{align*}
Thus \(G[f_\varepsilon]\to0\) provided that
\[
n-\alpha > \gamma(n-\alpha) + \frac{n}{t} - \alpha,
\]
or equivalently,
\[
\gamma > \frac{n}{t^\prime(n-\alpha)}.
\]
This is immediate since \(\gamma(q)\) is strictly decreasing in \(q\in(0,1)\), and
\[
\gamma\!\left(\frac{n}{\alpha}\right)
=
\frac{n}{t^\prime(n-\alpha)}.
\]
Therefore, we conclude that $\mathscr{C}(n,\alpha,q,t)=0$ in the case $0<q<\frac{n}{\alpha}$.

\textbf{Case 2.}
Assume that $q=\frac{n}{\alpha}<t<1$. 

Define
\[
f_k(x)
:=
\begin{cases}
e^{-\alpha},
&
|x|\le e,
\\[0.3em]
|x|^{-\alpha}
(\log|x|)^{-\frac{\alpha}{n}\left(1+\frac1k\right)},
&
|x|>e.
\end{cases}
\]
It is straightforward to verify that 
\[
0\not\equiv f_k\in L^1(\R^n)\cap L^q(\mathbb{R}^n).
\]
We claim that
$$
\lim\limits_{k\to\infty}F[f_k]=\lim\limits_{k\to\infty}\frac{\|E_\alpha f_k\|_{L^{t^\prime}(\R^n)}}{\|f_k\|_{L^1(\R^n)}^{\gamma}\,\|f_k\|_{L^q(\R^n)}^{1-\gamma}}=0.
$$
Indeed, elementary estimates show that
\[
\|f_k\|_{L^q(\mathbb{R}^n)}\to\infty
\quad \text{as } k\to\infty,
\]
and that there exists a constant $M>0$, independent of $k$, such that
\begin{equation}\label{pf-2.1(1)-}
\frac1M
\le
\int_{\mathbb{R}^n}f_k(y)\,dy
\le
M,
\ \ 
\int_{\mathbb{R}^n}
|y|^{\alpha-n}f_k(y)\,dy
\le
M.
\end{equation}
It remains to establish a uniform upper bound for $\|E_\alpha f_k\|_{L^{t^\prime}(\R^n)}$.

Using  \eqref{elementary inequality}, we estimate
\begin{align*}
E_\alpha f_k(x) 
&= \int_{\mathbb{R}^n} f_k(y)\,|x-y|^{\alpha-n} \, dy \\
&\le \max\{1,2^{\alpha-n-1}\}
\int_{\mathbb{R}^n} f_k(y)\,\bigl(|x|^{\alpha-n}+|y|^{\alpha-n}\bigr)\,dy \\
&\le \max\{1,2^{\alpha-n-1}\}\, M\, \bigl(|x|^{\alpha-n}+1\bigr).
\end{align*}
Recall the Beta function
\[
B(r,s)=\int_0^1 u^{r-1}(1-u)^{s-1}du,
\ \  r,s>0.
\]
A direct computation  yields 
\begin{equation}\label{Beta Type Function}
\int_0^\infty
v^{\lambda r-1}
(v^\lambda+a)^{-(r+s)}
\,dv
=
\frac{1}{\lambda a^s}
B(r,s),
\end{equation}
for all $a,\lambda,p,q>0$. 
Consequently, 
\begin{align*}
\|E_\alpha f_k\|_{L^{t^\prime}(\R^n)} 
&\le \max\{1,2^{\alpha-n-1}\} M
\Biggl(
\sigma(n)\int_0^\infty r^{n-1}\left(r^{\alpha-n}+1\right)^{t^\prime}\,dr
\Biggr)^{\frac{1}{t^\prime}}\\
&= \max\{1,2^{\alpha-n-1} \} M
\Biggl(
\frac{\sigma(n)}{\alpha-n}
\,
B\Bigl(\frac{n}{\alpha-n},-t^\prime-\frac{n}{\alpha-n}\Bigr)
\Biggr)^{\frac{1}{t^\prime}},
\end{align*}
where $\sigma(n)$ denotes the surface measure of the unit sphere in $\mathbb{R}^n$.
Since $\|E_\alpha f_k\|_{L^{t^\prime}(\R^n)}$ is uniformly bounded in \(k\) while \(\|f_k\|_{L^q(\R^n)}\to\infty\) as \(k\to\infty\), we conclude that
\(F[f_k]\to0\) as \(k\to\infty\).

\medskip

To prove \textit{(ii)}, it suffices by Lemma~\ref{lemma-rearrangement} to consider radially symmetric non-increasing functions.

We first assume $\gamma\neq0$ and establish a pointwise lower bound for $E_\alpha f$. 
We claim that there exists a constant $C>0$, independent of $f$ and $x$, such that
\begin{equation}\label{LB_E}
E_\alpha f(x) \ge C \|f\|_{L^{1}(\R^n)} |x|^{\alpha-n},
\quad \forall x \in \mathbb{R}^n\setminus\{0\},
\quad 0 \le f = f^* \in L^1(\mathbb{R}^n).
\end{equation}
In fact, decompose \(\mathbb{R}^n=\Omega_1\cup\Omega_2\cup\Omega_3\), where
\begin{align*}
\Omega_1 &:= \left\{ |y| < \frac{|x|}{2} \right\}, \\
\Omega_2 &:= \left\{ \frac{|x|}{2} \le |y| < \frac{3|x|}{2} \right\}, \\
\Omega_3 &:= \left\{ |y| \ge \frac{3|x|}{2} \right\}.
\end{align*}
Let
\[
I_i := \int_{\Omega_i} |x-y|^{\alpha-n} f(y)\,dy,
\ \  i=1,2,3.
\]
For $I_1$ and $I_3$, we directly have
\begin{equation}\label{SI_1}
I_1 \ge 2^{\,n-\alpha}|x|^{\alpha-n}\int_{\Omega_1} f(y)\,dy,
\ \ 
I_3 \ge 2^{\,n-\alpha}|x|^{\alpha-n}\int_{\Omega_3} f(y)\,dy .
\end{equation}
For $I_2$, set
\[
\Omega_{21}:=\left\{ \tfrac{|x|}{2}\le |y|<\tfrac{3|x|}{4} \right\},
\ \  
C(n):=\frac{|\Omega_2|}{|\Omega_{21}|}
      =\frac{6^{n}-2^{n}}{3^{n}-2^{n}} .
\]
Since $f$ is radially non-increasing, it follows that
\[
\int_{\Omega_2} f(y)\,dy
\le C(n)\int_{\Omega_{21}} f(y)\,dy .
\]
Moreover, for \(y\in\Omega_{21}\), we have \(|x-y|\ge |x|-|y|\ge \frac{|x|}{4}\). Hence
\begin{equation}\label{SI_2}
I_2
\ge 4^{\,n-\alpha}|x|^{\alpha-n}\int_{\Omega_{21}} f(y)\,dy
\ge \frac{4^{\,n-\alpha}}{C(n)}\,|x|^{\alpha-n}
   \int_{\Omega_2} f(y)\,dy .
\end{equation}
Combining \eqref{SI_1} and \eqref{SI_2}, we may take
\[
C:=\min\left\{2^{\,n-\alpha},\frac{4^{\,n-\alpha}}{C(n)}\right\},
\]
which yields \eqref{LB_E}.


On the other hand, by Lemma \ref{lemma-rearrangement} we know that $E_\alpha f^*$ is radial and strictly increasing. Hence,
\begin{equation}\label{es-rearrange}
E_\alpha f(x)
\ge
E_\alpha f(0)
=
\int_{\mathbb{R}^n} |y|^{\alpha-n} f(y)\,dy .
\end{equation}
Let \(\mathscr C_0=\frac12\min\{1,C\}\). 
Then
\begin{equation}    \label{es-lower-bound-Ef-Rn}
E_\alpha f(x)\ge
\mathscr C_0
\left(
\|f\|_{L^1(\mathbb{R}^n)}|x|^{\alpha-n}
+
\int_{\mathbb{R}^n}|y|^{\alpha-n}f(y)\,dy
\right).
\end{equation}
Consequently,
\[
\|E_\alpha f\|_{L^{t^\prime}(\R^n)}
\ge \mathscr C_0
\left( \int_{\mathbb{R}^n}
\left(
\|f\|_{L^{1}(\R^n)} |x|^{\alpha-n}
+ \int_{\mathbb{R}^n} |y|^{\alpha-n} f(y)\,dy
\right)^{t^\prime} dx
\right)^{\frac{1}{t^\prime}}.
\]
Applying the Carlson--Levin inequality \eqref{CLI} yields
\[
\scalebox{0.94}{$\displaystyle
\|E_\alpha f\|_{L^{t^\prime}(\R^n)}
\ge \mathscr C_0 \|f\|_{L^{1}(\R^n)}
\left(
\sigma(n)
\int_0^\infty r^{n-1}
\left(
r^{\alpha-n}
+
\mathscr L(n, \alpha-n,q)\left(
\frac{\|f\|_{L^{q}(\R^n)}}{\|f\|_{L^{1}(\R^n)}}
\right)^{\frac{q(\alpha-n)}{n(1-q)}}
\right)^{t^\prime} dr
\right)^{\frac{1}{t^\prime}}.
$}
\]
Using the integral identity \eqref{Beta Type Function}, it follows that
\[
\|E_\alpha f\|_{L^{t^\prime}(\R^n)}
\ge 
\mathscr{C}'
\, \|f\|_{L^{1}(\R^n)}^{
1 - \frac{q(\alpha-n)}{n(1-q)}
\left(1+\frac{n}{t^\prime(\alpha-n)}\right)
}
\, \|f\|_{L^{q}(\R^n)}^{
\frac{q(\alpha-n)}{n(1-q)}
\left(1+\frac{n}{t^\prime(\alpha-n)}\right)
}.
\]
Since
\[
1-\gamma
=
\frac{q(\alpha-n)}{n(1-q)}
\left(1+\frac{n}{t^\prime(\alpha-n)}\right),
\]
we conclude that
\[
\mathscr{C}(n,\alpha,q,t)\ge \mathscr{C}'>0.
\]
Moreover, the constant \(\mathscr{C}'\) is explicitly given by
\[
\mathscr{C}'
=\mathscr C_0 \left[
\frac{\sigma(n)}{\alpha - n}\, 
\bigl(\mathscr{L}(n,\alpha-n,q)\bigr)^{
t^\prime + \frac{n}{\alpha - n}}  \,
B\!\left(\frac{n}{\alpha - n},\, -t^\prime - \frac{n}{\alpha - n}\right)\, 
\right]^{\frac{1}{t^\prime}}.
\]

If $\gamma=0$, then the minimization problem \eqref{mini-problem-Rn} becomes
\begin{equation*}
\mathscr{C}(n,\alpha,q,t):=\inf\left\{F[f]:=\frac{\|E_\alpha f\|_{L^{t^\prime}(\R^n)}}{\|f\|_{L^q(\R^n)}}:0\le f\in L^q(\R^n), f \not\equiv 0 \right\}.
\end{equation*}
We can assume $f\in C_0^\infty(\R^n)$, argue as above and apply approximation process to end the proof.
\end{proof}

\begin{coro}\label{coro-extremal-functions}
Let $\alpha=n+2$ and $\frac{n}{n+2}<q,t<1$. 
Then the extremal functions for \eqref{ineq-main-Rn} are given by
\[
f(x) = \bigl(1+|x|^2\bigr)^{-\frac{1}{1-q}},
\]
up to translations, dilations, and constant multiples.
Moreover, the sharp constant is
\[
\scalebox{0.91}{$\displaystyle
\mathscr{C}(n,n+2,q,t)
=
\frac{1}{\pi}
\left(
\frac{n(1-q)}{q(n+2)-n}
\left(\frac{q(n+2)-n}{2q}\right)^{\frac{2}{n(1-q)}}
\left(\frac{\Gamma\!\left(\frac{1}{1-q}\right)}{\Gamma\!\left(\frac{1}{1-q}-\frac{n}{2}\right)}\right)^{\frac{2}{n}}
\right)^{1+\frac{n}{2t^\prime}}
\left(
\frac{\Gamma\!\left(-t^\prime-\frac{n}{2}\right)}{\Gamma(-t^\prime)}
\right)^{\frac{1}{t^\prime}}
$}.
\]
\end{coro}

\begin{proof}
It suffices to prove \eqref{ineq-main-Rn} for radially symmetric non-increasing functions. 
In this case,
\[
\int_{\mathbb{R}^n} y\,f(y)\,dy = 0.
\]
Consequently, the equality holds in \eqref{es-lower-bound-Ef-Rn} with $\mathscr C_0=1$, namely,
\[
E_2 f(x)
=\int_{\mathbb{R}^n}|x-y|^{2}f(y)\,dy
=\|f\|_{L^{1}(\R^n)}|x|^{2}
+\int_{\mathbb{R}^n}|y|^{2}f(y)\,dy .
\]
Therefore, by using the same arguments as the last part of Proposition \ref{Prop-SNC-Rn},  modulo translations, 
the equality in \eqref{ineq-main-Rn}
holds under the same conditions as the equality in the Carlson--Levin
inequality \eqref{CLI}.

Taking into account that 
\[
\mathscr{C}(n,n+2,q,t)
=
\left[
\frac{\sigma(n)}{2}
\bigl(\mathscr{L}(n,2,q)\bigr)^{\frac{n+2t^\prime}{2}}
B\!\left(\frac{n}{2},-t^\prime-\frac{n}{2}\right)
\right]^{\frac{1}{t^\prime}},
\]
where
\[
\mathscr L(n,2,q)
=
\frac{n(1-q)}{(n+2)q-n}
\left(
\frac{(n+2)q-n}
     {2^q q}
\right)^{\frac{2}{n(1-q)}}
\left(
\frac{
\Gamma\!\left(\frac1{1-q}\right)}
{
\pi^{n/2}
\Gamma\!\left(\frac1{1-q}-\frac n2\right)
}
\right)^{\frac{2}{n}}
\]
is the sharp constant in the Carlson--Levin inequality, we obtain the explicit expression for $\mathscr{C}(n,n+2,q,t)$.
\end{proof}

\subsection{Existence of extremal functions for reverse HLS inequality on \texorpdfstring{$\mathbb{R}^n$}{Rn}}
We next study the existence of minimizers for the minimization problem \eqref{mini-problem-Rn}. 

\begin{prop}\label{prop-Minimizer}   
Let $\frac{n}{\alpha}<t<1$, and $\frac{nt^\prime}{n+\alpha t^\prime}\le q<1$. 
Then the best constant $\mathscr{C}(n,\alpha,q,t)$ is achieved.
\end{prop}


\begin{proof}
If $q=\frac{nt^\prime}{n+\alpha t^\prime}$, then $\gamma=0$. This is the case due to \cite{2015-Beckner, 2015-DouZhu-IM-R}.

In the following we always assume that $\frac{nt^\prime}{n+\alpha t^\prime}<q<1$. Then $\gamma<0$.

We consider the minimization problem \eqref{mini-problem-Rn} and let $\{f_j\}$ be a minimizing sequence, that is,
\[
\lim_{j\to\infty}F[f_j]=\mathscr{C}(n,\alpha,q,t).
\]
For convenience, we divide the proof into three steps.

\medskip
\textbf{\textit{Step 1.}} Construction of a candidate minimizer.

By the rearrangement inequality \eqref{ineq-rearrangement-Rn}, we have
$$
F[f_j^*]=\frac{\|E_\alpha f_j^*\|_{L^{t^\prime}(\R^n)}}{\|f_j^*\|^\gamma_{L^1(\R^n)}\|f_j^*\|^{1-\gamma}_{L^q(\R^n)}}
\le F[f_j].
$$
So that $\{f_j^*\}$ is still a minimizing sequence. 
Hence, without loss of generality, we may assume that each $f_j$ is nonnegative, radially symmetric, and non-increasing. 
For simplicity of notation, we write $f_j(x)=f_j(|x|)$.

Moreover, by the scaling invariance and homogeneity of $F$, we may normalize the sequence to satisfy
\[
\|f_j\|_{L^1(\R^n)}=\|f_j\|_{L^q(\R^n)}=1,
\ \  \forall j\in\mathbb{N}.
\]
From the normalization $\|f_j\|_{L^q(\R^n)}=1$, it follows that for any $R>0$,
\begin{equation*}
1=\sigma_n\int_0^\infty f_j(r)^qr^{n-1}\,dr
\ge\sigma_n\int_0^R f_j(r)^qr^{n-1}\,dr 
\ge \frac{\sigma_n}{n}f_j(R)^qR^n.
\end{equation*} 
Consequently, we obtain the pointwise bound
\begin{equation}\label{e-ms-1} 
f_j(r)\le C\min\left\{r^{-\frac{n}{q}},\, r^{-n}\right\},
\ \  \forall r>0,
\end{equation}
where $C>0$ is independent of $j$.

By Helly's selection theorem, up to a subsequence, we have $f_j\to f_0$ almost everywhere for some measurable function $f_0$. 
Moreover, $f_0$ is nonnegative, radially symmetric, and non-increasing, and satisfies the same pointwise bound \eqref{e-ms-1}.

\medskip
\textbf{\textit{Step 2.}} Proof of the inequality
\begin{equation}\label{es-pf-prop-attained-Rn-Ef}
\|E_\alpha f_0\|_{L^{t^\prime}(\R^n)}
\le
\lim_{j\to\infty}\|E_\alpha f_j\|_{L^{t^\prime}(\R^n)}.
\end{equation}

Since $\|E_\alpha f_j\|_{L^{t^\prime}(\R^n)}\to \mathscr C(n,\alpha,q,t)>0$, we may assume that there exists a constant $C>0$, independent of $j$, such that
\[
\|E_\alpha f_j\|_{L^{t^\prime}(\R^n)}^{\,t^\prime} \le C,
\ \  \forall j \in\mathbb{N}.
\]
By Lemma~\ref{lemma-rearrangement}, each $E_\alpha f_j$ is radially symmetric and strictly increasing.
Consequently, 
\[
(E_\alpha f_j(x))^{t^\prime}\le C|x|^{-n},
\ \  \forall x\in \R^n\setminus\{0\}.
\]
Thus, up to a subsequence, Helly's selection theorem implies that 
\[
(E_\alpha f_j(|x|))^{t^\prime}
\to
g(|x|)
\quad \text{a.e. in } \R^n
\]
for some measurable function $g$.
Since $E_\alpha f_j(x)\to g(x)^{1/t^\prime}$ almost everywhere on $\R^n$,
then by Fatou's lemma,
\[
E_\alpha f_0(x)
=\int_{\R^n}|x-y|^{\alpha-n}\lim_{j\to\infty} f_j(y)\,dy
\le \liminf_{j\to\infty} E_\alpha f_j(x)
= \lim_{j\to\infty} E_\alpha f_j(x).
\]
Thus
\[
(E_\alpha f_0(x))^{t^\prime}
\ge
\lim_{j\to\infty}(E_\alpha f_j(x))^{t^\prime}
=
g(x).
\]
On the other hand, by \eqref{es-lower-bound-Ef-Rn} and the Carlson--Levin inequality, we have
\[
E_\alpha f_j(x)
\ge
\mathscr C_0 \Bigl(|x|^{\alpha-n}
+
\mathscr L(n,\lambda,q)\Bigr)
=:M(x),
\ \  \forall j \in\mathbb{N}.
\]
Since $t^\prime<\frac{n}{n-\alpha}<0$, it follows that
\[
(E_\alpha f_j(x))^{t^\prime}\le (M(x))^{t^\prime}\in L^1(\R^n).
\]
Therefore, by the dominated convergence theorem,
\[
\lim_{j\to\infty}\int_{\R^n}(E_\alpha f_j(x))^{t^\prime}\,dx
=
\int_{\R^n} g(x)\,dx
\le
\int_{\R^n}(E_\alpha f_0(x))^{t^\prime}\,dx.
\]
Since $t^\prime<0$, this yields
$$
\|E_\alpha f_0\|_{L^{t^\prime}(\mathbb{R}^n)}\le\lim_{j\to\infty} \|E_\alpha f_j\|_{L^{t^\prime}(\mathbb{R}^n)}.
$$

\medskip
\textbf{\textit{Step 3.}} Attainment of the infimum.

By Fatou's lemma, 
\[
\|f_0\|_{L^1(\R^n)} \le \liminf_{j\to\infty}\|f_j\|_{L^1(\R^n)} =1.
\]
Since $\gamma<0$, it remains to show that $\|f_0\|_{L^q(\R^n)}=1$. 
By \eqref{es-rearrange}, for any fixed $j$, one has  
\[
\int_{\mathbb{R}^n} |y|^{\alpha-n} f_j(y)\,dy<\infty
\]
since $\{f_j\}$ is a minimizing sequence. Then Carlson--Levin inequality \eqref{CLI} implies that $f_j\in L^s(\mathbb{R}^n)$
for any $s\in\left( \frac{n}{\alpha},q \right)$. Thus by using the inequality \eqref{ineq-main-Rn},
$$
\|E_\alpha f_j\|_{{L^{t^\prime}(\mathbb{R}^n)}} 
\ge 
\mathscr{C}^\prime \|f_j\|_{L^s(\mathbb{R}^n)}^{1-\gamma(s)} ,
$$
where 
$$
1-\gamma(s)=1-\frac{n-s\alpha-\frac{n}{t^\prime}s}{n(1-s)}>0. 
$$
This implies that $\{ f_j \}$ is uniformly bounded in $L^s(\R^n)$. 
In particular, 
\[
f_j(|x|)\le C|x|^{-\frac{n}{s}},
\ \  \forall j \in\mathbb{N},
\]
for some constant $C>0$ independent of $j$.
Combining this with the estimate $(\ref{e-ms-1})$, we deduce that
$$ 
f_j(|x|)\le C \min\left\{|x|^{-\frac{n}{s}},|x|^{-n}\right\}\in L^q(\R^n).
$$
Therefore, by the dominated convergence theorem, 
\[
\|f_0\|_{L^q(\R^n)}=\lim_{j\to\infty}\|f_j\|_{L^q(\R^n)}=1.
\]

Finally, by the definition of $\mathscr{C}(n,\alpha,q,t)$ and Step 2, 
$$
\mathscr{C}(n,\alpha,q,t) 
= 
\lim\limits_{j\to\infty}\|E_\alpha f_j\|_{L^{t^\prime}(\R^n)}
\ge 
\|E_\alpha f_0\|_{L^{t^\prime}(\R^n)}
\ge
\frac{\|f_0\|^{-\gamma}_{L^{1}(\R^n)}\|E_\alpha f_0\|_{L^{t^\prime}(\R^n)}}{{\|f_0\|^{1-\gamma}_{L^{q}(\R^n)}}} 
\ge 
\mathscr{C}(n,\alpha,q,t).
$$
Therefore,
\[
F[f_0]=\mathscr{C}(n,\alpha,q,t),
\]
and the infimum is achieved. 
\end{proof}

\begin{rem}\label{rem-mini-sequence}
Let $\{f_j\}$ be a minimizing sequence for the minimization problem \eqref{mini-problem-Rn}. It is obvious that, without loss of generality, we can assume that each $f_j$ is bounded and has compact support. 
\end{rem}

\begin{rem}\label{rem-semi-continuous}
Similarly to \cite{2019-Carrillo-Delgadino-Dolbeault-Frank-Hoffmann-JMPA}, it is easy to see that the map $q\mapsto \mathscr{C}(n,\alpha,q,t)$ is upper semi-continuous. In particular, 
$$ \lim_{q\to\left(n/\alpha\right)^+} \mathscr{C}(n,\alpha,q,t) = 0 .$$
\end{rem}

\begin{prop}\label{prop-positivity-Rn}
Let $1\le n<\alpha$ and $\frac{n}{\alpha}<q,t<1$. 
If $f\ge0$ is an extremal function for $\mathscr C(n,\alpha,q,t)$, then $f$ is radial (up to a translation), monotone non-increasing and positive almost everywhere in $\mathbb R^n$.
\end{prop}

\begin{proof}
Since $\mathscr C(n,\alpha,q,t)>0$, we have $f\not\equiv0$. 
By rearrangement inequalities \eqref{ineq-rearrangement-Rn} and up to a translation, we know that $f$ is radial and monotone non-increasing. 
Assume by contradiction that $f=0$ on a measurable set $E\subset\mathbb R^n$ with $0<|E|<\infty$. 
Replacing $E$ by a bounded subset of sufficiently small positive measure if necessary, we may assume that $E$ is bounded and there exist positive constants $C_1,C_2$ such that
\[
E_\alpha \mathbf 1_E(x)
\le C_1 (1+|x|)^{\alpha-n}
\le C_2 E_\alpha f(x)
\ \  \text{for all }x\in\mathbb R^n.
\]

For $\varepsilon>0$ small, let $f_\varepsilon:=f+\varepsilon \mathbf 1_E$. Since $f=0$ on $E$, we have
\[
\|f_\varepsilon\|_{L^1(\mathbb R^n)}
=\|f\|_{L^1(\mathbb R^n)}+\varepsilon |E|,
\ \ 
\|f_\varepsilon\|_{L^q(\mathbb R^n)}^q
=\|f\|_{L^q(\mathbb R^n)}^q+|E|\varepsilon^q,
\]
and hence
\begin{equation}\label{es-pf-prop-f>0-1}
\|f_\varepsilon\|_{L^1(\mathbb R^n)}^\gamma \|f_\varepsilon\|_{L^q(\mathbb R^n)}^{1-\gamma}
=
\|f\|_{L^1(\mathbb R^n)}^\gamma \|f\|_{L^q(\mathbb R^n)}^{1-\gamma}
\bigl(1+O(\varepsilon)\bigr)
\biggl(
1+\frac{1-\gamma}{q}\frac{|E|}{\|f\|_{L^q(\mathbb R^n)}^q}\varepsilon^q+o(\varepsilon^q)
\biggr)
\end{equation}
as $\varepsilon\to0^+$.

Since $E_\alpha f_\varepsilon=E_\alpha f+\varepsilon E_\alpha \mathbf 1_E$, the mean value theorem yields
\[
\frac{(E_\alpha f_\varepsilon)^{t^\prime}-(E_\alpha f)^{t^\prime}}{\varepsilon}
=
t^\prime\bigl(E_\alpha f+\theta_\varepsilon(x) E_\alpha \mathbf 1_E\bigr)^{t^\prime-1}E_\alpha \mathbf 1_E
\]
for some $\theta_\varepsilon(x)\in(0,\varepsilon)$. Since $t^\prime<0$ and $E_\alpha \mathbf 1_E\ge0$, we have
\[
\bigl(E_\alpha f+\theta_\varepsilon(x) E_\alpha \mathbf 1_E\bigr)^{t^\prime-1}\le (E_\alpha f)^{t^\prime-1},
\]
and therefore
\[
\left|
\frac{(E_\alpha f_\varepsilon)^{t^\prime}-(E_\alpha f)^{t^\prime}}{\varepsilon}
\right|
\le |t^\prime|(E_\alpha f)^{t^\prime-1}E_\alpha \mathbf 1_E
\le C_2|t^\prime|(E_\alpha f)^{t^\prime}.
\]
Then
\[
\int_{\mathbb R^n}(E_\alpha f_\varepsilon)^{t^\prime}\,dx
=
(1+O(\varepsilon)\bigr)
\int_{\mathbb R^n}(E_\alpha f)^{t^\prime}\,dx .
\]
It follows that
\begin{equation}\label{es-pf-prop-f>0-2}
\|E_\alpha f_\varepsilon\|_{L^{t^\prime}(\mathbb R^n)}
=
\|E_\alpha f\|_{L^{t^\prime}(\mathbb R^n)}\bigl(1+O(\varepsilon)\bigr).
\end{equation}

Combining estimates \eqref{es-pf-prop-f>0-1} and \eqref{es-pf-prop-f>0-2}, we obtain
\[
F[f_\varepsilon]
=
F[f]\,
\frac{1+O(\varepsilon)}
{1+\frac{1-\gamma}{q}\frac{|E|}{\|f\|_{L^q(\mathbb R^n)}^q}\varepsilon^q+o(\varepsilon^q)}.
\]
Since $0<q<1$, one has $O(\varepsilon)=o(\varepsilon^q)$ as $\varepsilon\to0^+$. Therefore,
\[
F[f_\varepsilon]
=
F[f]
\left(
1-\frac{1-\gamma}{q}\frac{|E|}{\|f\|_{L^q(\mathbb R^n)}^q}\varepsilon^q+o(\varepsilon^q)
\right),
\]
which implies $F[f_\varepsilon]<F[f]$ for   $\varepsilon>0$ sufficiently small. 
This is a contradiction. Hence $f>0$ almost everywhere in $\mathbb R^n$.
\end{proof}



\subsection{Inequalities for \texorpdfstring{$q\ge1$}{q>1} on \texorpdfstring{$\mathbb{R}^n$}{Rn}} \label{sec-q>1}
We first consider the case $q>1$. As in the case $0<q<1$, the exponent $\gamma$ is given by 
\[
\gamma=\gamma(q):=\frac{q\alpha+\frac{n}{t^\prime}q-n}{n(q-1)},
\]
which takes values in $(1+\frac{\alpha}{n}-\frac{1}{t},+\infty)$.

\begin{proof}[\textbf{Proof of Theorem \ref{them-Rn-q-bigger-1}}]
Without loss of generality, we assume that $\|E_\alpha f\|_{L^{t^\prime}(\R^n)}<\infty$ for nonnegative function $0\not\equiv f\in L^1(\R^n) \cap L^q(\R^n)$.   
By \eqref{es-rearrange}, one has 
\[
\int_{\mathbb{R}^n} |y|^{\alpha-n} f(y)\,dy<\infty.
\]
Then Carlson--Levin inequality \eqref{CLI} implies that $f\in L^s(\mathbb{R}^n)$ for any $s\in\left( \frac{n}{\alpha},q \right)$. 
By Theorem \ref{them-Rn},
$$
\|E_\alpha f\|_{L^{t^\prime}(\R^n)}
\geq
\mathscr{C}(n,\alpha,s,t)\|f\|^{\gamma(s)}_{L^1( \R^n)}\|f\|^{1-\gamma(s)}_{L^s(\R^n)},
$$
where $$\gamma(s):=\frac{s\alpha+\frac{n}{t^\prime}s-n}{n(s-1)}.$$ 
Since $s<1<q$, H\"older's inequality yields
$$
\|f\|^{s}_{L^s(\R^n)}\|f\|^{q\frac{1-s}{q-1}}_{L^q(\R^n)}\ge \|f\|^{\frac{q-s}{q-1}}_{L^1(\R^n)}.
$$
Observe that
$$
\gamma(q)-1=(1-\gamma(s))\frac{q(1-s)}{s(q-1)}.
$$
 Hence, 
\begin{equation}\label{ineq-temp-index-s}  
\|E_\alpha f\|_{L^{t^\prime}(\R^n)} \|f\|^{\gamma(q)-1}_{L^q(\R^n)} \geq \mathscr{C}(n,\alpha,s,t)\|f\|^{\gamma(q)}_{L^1( \R^n)}.
\end{equation}
This proves \eqref{ineq-main-Rn-q-bigger-1} for some constant
\begin{equation} \label{ineq-const-lower-bound-Rn}
    \mathscr{C}(n,\alpha,q,t)\ge\mathscr{C}(n,\alpha,s,t).
\end{equation}

The existence of a radial non-increasing minimizer follows as in the proof of Proposition~\ref{prop-Minimizer}.

It remains to show that every minimizer has compact support.
The Euler--Lagrange equation for the minimizing problem \eqref{mini-problem-Rn}  for $q>1$  is
\begin{equation} \label{eq-EL-Rn-q>1}
   \frac{E_\alpha\bigl((E_\alpha f)^{t^\prime-1}\bigr)}
{\|E_\alpha f\|_{L^{t^\prime}(\R^n)}^{t^\prime}}
-\gamma\,\frac{1}{\|f\|_{L^1(\R^n)}}
+(\gamma-1)\,\frac{f^{q-1}}{\|f\|_{L^q(\R^n)}^{q}}
=0 
\end{equation}
in the interior of the support of $f\ge0$.
Therefore,
\[
f(x) = \left(
C_1 - C_2 E_\alpha\bigl((E_\alpha f)^{t^\prime-1}\bigr)(x)
\right)_+^{\frac{1}{q-1}  }
\]
for some positive constants $C_1,C_2$.
Moreover, since $0<\|E_\alpha f\|_{L^{t^\prime}(\R^n)}<\infty$, 
\[
E_\alpha\bigl((E_\alpha f)^{t^\prime-1}\bigr)(x)
=
\int_{\R^n}|x-y|^{\alpha-n}(E_\alpha f(y))^{t^\prime-1}\,dy
\sim
|x|^{\alpha-n}
\int_{\R^n}(E_\alpha f(y))^{t^\prime-1}\,dy
\]
as $|x|\to\infty$.
Hence
\[
C_1-C_2E_\alpha\bigl((E_\alpha f)^{t^\prime-1}\bigr)(x)<0
\]
for sufficiently large $|x|$.
Consequently, $f(x)=0$ for sufficiently large $|x|$, and therefore $f$ has compact support.
\end{proof}

\begin{rem}
Using lower bound \eqref{es-lower-bound-Ef-Rn} together with the Carlson--Levin inequality \eqref{CLI+}, one can also derive the inequality \eqref{ineq-main-Rn-q-bigger-1}, and
\[
\mathscr{C}(n,\alpha,q,t)
\ge
\mathscr C_0 \left[
\frac{\sigma(n)}{\alpha - n}\, 
\bigl(\mathscr{L}(n,\alpha-n,q)\bigr)^{
t^\prime + \frac{n}{\alpha - n}}  \,
B\!\left(\frac{n}{\alpha - n},\, -t^\prime - \frac{n}{\alpha - n}\right)\, 
\right]^{\frac{1}{t^\prime}}
>0.
\]
Moreover, if $\alpha=n+2$ and $\frac{n}{\alpha}<t<1<q$, then the equality holds in \eqref{ineq-main-Rn-q-bigger-1} if and only if
\[
f(x)=(1-|x|^2)_+^{\frac{1}{q-1}},
\]
up to translations, dilations, and multiplication by a positive constant. 
The corresponding sharp constant is given by
\[
\scalebox{0.89}{$\displaystyle
\mathscr{C}(n,n+2,q,t)
=
\frac{1}{\pi}
\left(
\frac{n(q-1)}{2q+n(q-1)}
\left(\frac{2q}{2q+n(q-1)}\right)^{\frac{2}{n(q-1)}}
\left(\frac{\Gamma\!\left(\frac{q}{q-1}+\frac{n}{2}\right)}{\Gamma\!\left(\frac{q}{q-1}\right)}\right)^{\frac{2}{n}}
\right)^{1+\frac{n}{2t^\prime}}
\left(
\frac{\Gamma\!\left(-t^\prime-\frac{n}{2}\right)}{\Gamma(-t^\prime)}
\right)^{\frac{1}{t^\prime}}
.
$}
\]
\end{rem}

\begin{proof}[\textbf{Proof of Theorem \ref{them-Rn-q=1-log}}]


We adapt the argument of \cite{2003-DelPino-Dolbeault-JFA}, which is based on an approximation via inequalities with \(q>1\).

Define
\[
J_{1}[f]
:=
\|E_\alpha f\|_{L^{t^\prime}(\R^n)}
\exp\left\{
\left(\frac{\alpha}{n}-\frac{1}{t}\right)
\int_{\R^n} f\log f\,dx
\right\}.
\]
For \(q>1\), set
\[
J_{q}[f]
:=
\|E_\alpha f\|_{L^{t^\prime}(\R^n)}
\left(
\int_{\R^n} f^q\,dx
\right)^{
\frac{1}{q-1}
\left(\frac{\alpha}{n}-\frac{1}{t}\right)
}.
\]
Let
\[
\mathscr A
:=
\left\{
f\ge0:
\int_{\R^n} f(x)\,dx=1,\ 
f\log f\in L^1(\R^n)
\right\}.
\]
We then consider the minimization problem
\begin{equation}\label{mini-problem-log-Rn}
\mathscr C(n,\alpha,1,t)
:=
\inf_{f\in\mathscr A} J_1[f].
\end{equation}

It suffices by Lemma~\ref{lemma-rearrangement} to consider radially symmetric non-increasing functions. 
Without loss of generality, we may further assume that \(f\in \mathscr A\cap L^\infty(\R^n)\). 
Actually, set
\[
g_j(x):=\min\{f(x),j\}, \quad f_j(x):=\frac{g_j(x)}{\|g_j\|_{L^1(\R^n)}}.
\]
One can check that $f_j\in\mathscr{A}\cap L^\infty(\R^n)$ and  $\lim\limits_{j\to\infty}J_1[f_j]=J_1[f]$.

By \eqref{ineq-main-Rn-q-bigger-1}, we get
\begin{equation}
\label{vari-ineq-main-Rn-q-bigger-1}
J_q[f]\ge \mathscr C(n,\alpha,q,t).
\end{equation}
Let
\[
F(q):=\int_{\R^n} f^q\,dx.
\]
For \(q\) close to \(1\) we have
\[
|f^q\log f|
\le
\|f\|^{q-1}_{L^\infty(\R^n)}|f\log f|
\in L^1(\R^n).
\]
Therefore, 
\[
F'(q)=\int_{\R^n}f^q\log f\,dx.
\]
Since \(F(1)=1\), it follows that
\[
\lim_{q\to1^+}
\frac{1}{q-1}\log\int_{\R^n}f^q\,dx
=
\lim_{q\to1^+}
\frac{\log F(q)-\log F(1)}{q-1}
=
\int_{\R^n}f\log f\,dx.
\]
Thus, letting \(q\to1^+\) in \eqref{vari-ineq-main-Rn-q-bigger-1}, we obtain
\[
J_1[f]
=
\lim_{q\to1^+}J_q[f]
\ge
\lim_{q\to1^+}\mathscr C(n,\alpha,q,t).
\]
Combining with \eqref{ineq-const-lower-bound-Rn}, for $q>1$ close to $1$, there exists a constant \(C_1>0\) such that
\begin{equation} \label{estimate-lower-upper-C}
\frac{1}{C_1}\le \mathscr C(n,\alpha,q,t)\le C_1.
\end{equation}
Then by the definition of $\mathscr C(n,\alpha,1,t)$, we have
\begin{equation}\label{estimate-const-q=1}
\mathscr C(n,\alpha,1,t)\ge\lim_{q\to1^+}\mathscr C(n,\alpha,q,t)>0.    
\end{equation}

Let \(f_j\) be an extremal function for \eqref{ineq-main-Rn-q-bigger-1} with \(q=q_j>1\), and $\int_{\R^n} f_j(x)\,dx=1,$ namely,
\[
J_{q_j}[f_j]
=
\mathscr C(n,\alpha,q_j,t),
\]
where each \(f_j\) is radially symmetric and non-increasing.
Without loss of generality, we may normalize it by
\[
\int_{B_1} f_j(x)\,dx=\frac12.
\]
Let $q_j\to 1$ as $j\to+\infty$.
By Helly's selection theorem, there exists a nonnegative radially non-increasing function \(f_0\in L^1(\R^n)\) such that, up to a subsequence,
\[
f_j\to f_0 \quad \text{a.e. in } \R^n.
\]
To prove $J_1[f_0]=\mathscr C(n,\alpha,1,t)$, we need following claims.

\medskip
\textbf{Claim 1}. There exists a constant \(C_2>0\), independent of \(j\), such that
\begin{equation}\label{estimate-lower-upper-Ef}
 \frac{1}{C_2}
\le
\|E_\alpha f_j\|_{L^{t^\prime}(\R^n)}
\le
C_2.
\end{equation}

If $\|E_\alpha f_j\|_{L^{t^\prime}(\R^n)}\to\infty$, then 
\[
\left(
\int_{\R^n} f_j^{q_j}\,dx
\right)^{
\frac{1}{q_j-1}}
\to0.
\]
Moreover, from the Euler--Lagrange equation \eqref{eq-EL-Rn-q>1},
\[
f_j(x)
=
\left\{
\left(\int_{\R^n} f_j^{q_j}\,dx\right)
\left[
1+
\frac{q_j-1}{
q_j\left(\frac{\alpha}{n}-\frac1t\right)
}
-
\frac{
q_j-1
}{
q_j\left(\frac{\alpha}{n}-\frac1t\right)
}
\frac{
E_\alpha\!\left((E_\alpha f_j)^{t^\prime-1}\right)(x)
}{
\int_{\R^n}(E_\alpha f_j)^{t^\prime}\,dx
}
\right]
\right\}_+^{\frac{1}{q_j-1}} .
\]
Hence
\begin{equation}\label{estimate-from-EL-1}
\|f_j\|_{L^\infty(\R^n)}
\le
\left(\int_{\R^n} f_j^{q_j}\,dx\right)^{\frac{1}{q_j-1}}
\left[
1+
\frac{q_j-1}{
q_j\left(\frac{\alpha}{n}-\frac1t\right)}
\right]^{\frac{1}{q_j-1}} .
\end{equation}
Since
\[
\left[
1+
\frac{q_j-1}{
q_j\left(\frac{\alpha}{n}-\frac1t\right)}
\right]^{\frac{1}{q_j-1}}
\to
\exp\left\{
\frac{1}{\frac{\alpha}{n}-\frac1t}
\right\},
\]
as $q_j\to1$, we obtain
\[
\|f_j\|_{L^\infty(\R^n)}\to0.
\]
This contradicts the normalization
\[
\frac12
=
\int_{B_1} f_j(x)\,dx
\le
|B_1|\,\|f_j\|_{L^\infty(\R^n)}
\to0.
\]

Then $\|E_\alpha f_j\|_{L^{t^\prime}(\R^n)}\le C$ for some $C>0$. And therefore for some fixed \(s_0\in \left(\frac{n}{\alpha},1\right)\), it follows from \eqref{ineq-main-Rn} that
\[
C
\ge
\|E_\alpha f_j\|_{L^{t^\prime}(\R^n)}
\ge
\mathscr C(n,\alpha,s_0,t)
\|f_j\|_{L^1(\R^n)}^{\gamma(s_0)}
\|f_j\|_{L^{s_0}(\R^n)}^{1-\gamma(s_0)}.
\]
Since \(\gamma(s_0)<1\), we deduce that \(\{f_j\}\) is uniformly bounded in \(L^{s_0}(\R^n)\). 
On the other hand, because of
\[
f_j(x)
\le
f_j(1)
\le
\frac{1}{|B_1|}
\int_{B_1} f_j(y)\,dy
=
\frac{1}{2|B_1|},\ \ \forall |x|\ge 1.
\]
We have 
\[
\|f_j\|_{L^{s_0}(\R^n)}^{s_0}
\ge
\int_{\R^n\setminus B_1} f_j(x)^{s_0}\,dx
\ge
\left(\frac{1}{2|B_1|}\right)^{s_0-1}
\int_{\R^n\setminus B_1} f_j(x)\,dx
=
2^{-s_0}|B_1|^{\,1-s_0}
>0.
\]
Thus there exists a constant \(C_2>0\), independent of \(j\), such that \eqref{estimate-lower-upper-Ef} holds.

\medskip
\textbf{Claim 2.}  We claim that $f_0\in\mathscr{A}$, i.e. 
\[
\int_{\R^n} f_0(x)\,dx =1,\ \  f_0\log f_0\in L^1(\R^n).
\]

Recalling \eqref{estimate-lower-upper-C} and \eqref{estimate-lower-upper-Ef}, we obtain
\[
\frac{1}{C_1C_2}
\le
\left(
\int_{\R^n} f_j^{q_j}\,dx
\right)^{\frac{1}{q_j-1}
\left(\frac{\alpha}{n}-\frac{1}{t}\right)}
\le
C_1C_2.
\]
Combining this with \eqref{estimate-from-EL-1}, we conclude that
\[
\|f_j\|_{L^\infty(\R^n)}
\le C_3.
\]
Thus, we deduce that \(\{f_j\}\) is uniformly bounded in \(L^{q_0}(\R^n)\) for some $q_0>1$. 
Thus we have 
$$ 
f_j(|x|)\le C_4 \min\left\{|x|^{-\frac{n}{s}},|x|^{-\frac{n}{q_0}}\right\}\in L^1(\R^n).
$$
Consequently, the dominated convergence theorem implies that
\[
\|f_0\|_{L^1(\R^n)}
=
\lim_{j\to\infty}
\|f_j\|_{L^1(\R^n)}
=
1.
\]



Consider
\[
|f_0\log f_0|=(f_0\log f_0)_++(f_0\log f_0)_-.
\]
By the uniform \(L^\infty\)-bound on \(f_j\),
\[
0\le f_0\le C
\ \ \text{a.e. in }\R^n.
\]
Therefore, 
\[
(f_0\log f_0)_+
\le
(\log C)_+ f_0
\in L^1(\R^n).
\]
On the other hand, for \(s_0<s<1\), there exists a constant \(C_s>0\) such that
\[
0\le -r\log r\le C_s r^s,
\ \  0<r\le1.
\]
Hence
\[
(f_0\log f_0)_-
\le
C_s f_0^s
\in L^1(\R^n).
\]
Combining the estimates for the positive and negative parts of $f_0\log f_0$, we conclude that
\[
f_0\log f_0\in L^1(\R^n).
\]

\textbf{Claim 3.} We have the inequality $\lim\limits_{j\to\infty} J_{q_j} [f_j] \ge J_1 [f_0].$

For each \(j\), since \(\int_{\R^n}f_j\,dx=1\), we have
\[
\frac{1}{q_j-1}
\log\int_{\R^n} f_j^{q_j}\,dx
=
\frac{\log\int_{\R^n} f_j^{q_j}\,dx-\log\int_{\R^n} f_j\,dx}{q_j-1}.
\]
The mean value theorem gives some \(\theta_j\in(1,q_j)\) such that
\[
\frac{1}{q_j-1}
\log\int_{\R^n} f_j^{q_j}\,dx
=
\frac{
\int_{\R^n} f_j^{\theta_j}\log f_j\,dx
}{
\int_{\R^n} f_j^{\theta_j}\,dx
}.
\]
The dominated convergence theorem implies that as $j\to +\infty$,
\[
f_j^{\theta_j}\to f_0
\ \ \text{in }L^1(\R^n),
\]
and
\[
f_j^{\theta_j}\log f_j
\to
f_0\log f_0
\ \ \text{in }L^1(\R^n).
\]
Hence
\[
\lim_{j\to\infty}
\frac{1}{q_j-1}
\log\int_{\R^n} f_j^{q_j}\,dx
=
\int_{\R^n} f_0\log f_0\,dx.
\]
Combining this with \eqref{es-pf-prop-attained-Rn-Ef}, we conclude that 
\[
\lim_{j\to\infty} J_{q_j}[f_j]
\ge
\|E_\alpha f_0\|_{L^{t^\prime}(\R^n)}
\exp\left\{
\left(\frac{\alpha}{n}-\frac{1}{t}\right)
\int_{\R^n} f_0\log f_0\,dx
\right\}
=
J_1[f_0].
\]

\medskip
By using the above claims, the definition of $\mathscr C(n,\alpha,1,t)$ and  \eqref{estimate-const-q=1}, we have 
\[
J_1 [f_0]=\mathscr C(n,\alpha,1,t).
\] 

\medskip
We now show some basic properties of extremal functions.

Let \(f\) be an extremal function for \eqref{ineq-main-Rn-log}. 
By the rearrangement inequality, we may assume, up to a translation, that \(f\) is radially symmetric and non-increasing. 
More precisely, if \(f^*\) denotes the symmetric decreasing rearrangement of \(f\), then
\[
\|E_\alpha f^*\|_{L^{t^\prime}(\R^n)}
\le
\|E_\alpha f\|_{L^{t^\prime}(\R^n)},
\ \ 
\int_{\R^n} f^*\log f^*\,dx
=
\int_{\R^n} f\log f\,dx.
\]
Hence \(f^*\) is also an extremal function. 
In particular, every extremal function is, up to translation, radially symmetric and non-increasing. 

\textbf{Claim 4.} The extremal function \(f>0\) almost everywhere on \(\R^n\).

Since \(\mathscr C(n,\alpha,1,t)>0\), we have \(f\not\equiv0\). 
Assume by contradiction that \(f=0\) on a measurable set \(E\subset\R^n\) with \(0<|E|<\infty\). 
Replace \(E\) by a bounded subset of positive measure if necessary.
For \(\varepsilon>0\) small, set
\[
f_\varepsilon
:=
\frac{f+\varepsilon\mathbf 1_E}{1+\varepsilon |E|}.
\]
We first compute the entropy term. Since \(f=0\) on \(E\),
\begin{align*}
\int_{\R^n} f_\varepsilon\log f_\varepsilon\,dx
&=
\frac{1}{1+\varepsilon |E|}
\int_{\R^n\setminus E}
f\log\left(
\frac{f}{1+\varepsilon |E|}
\right)\,dx
+
\frac{\varepsilon}{1+\varepsilon |E|}
\int_E
\log\left(
\frac{\varepsilon}{1+\varepsilon |E|}
\right)\,dx
\\
&=
\int_{\R^n} f\log f\,dx + |E|\varepsilon\log\varepsilon+O(\varepsilon).
\end{align*}
Next, similar as \eqref{es-pf-prop-f>0-2}, we have
\[
\log\|E_\alpha f_\varepsilon\|_{L^{t^\prime}(\mathbb R^n)}
=
\log\|E_\alpha f\|_{L^{t^\prime}(\mathbb R^n)}+O(\varepsilon).
\]
Combining the above two expansions, we obtain
\[
\log J_1[f_\varepsilon]-\log J_1[f]
=
\left(\frac{\alpha}{n}-\frac1t\right)
\varepsilon |E|\log\varepsilon + O(\varepsilon).
\]
Thus
\[
\log J_1[f_\varepsilon]-\log J_1[f]<0
\]
for all \(\varepsilon>0\) sufficiently small.  
Hence $J_1[f_\varepsilon]<J_1[f]$, which contradicts the minimality of $f$.

Finally, consider the Euler--Lagrange equation associated with the constrained minimization problem \eqref{mini-problem-log-Rn}.
There exists a Lagrange multiplier \(\lambda\in\R\) such that
\[
\left(\frac{\alpha}{n}-\frac1t\right)
\bigl(\log f(x)+1\bigr)
+
\frac{
E_\alpha\!\left((E_\alpha f)^{t^\prime-1}\right)(x)
}{
\int_{\R^n}(E_\alpha f)^{t^\prime}\,dx
}
=
\lambda.
\]
This equation follows from the fact that $f$ is positive almost everywhere according to Claim~4.
Equivalently, after absorbing constants into the normalization,
\[
f(x)
=
A\exp\left\{
-\frac{1}{\frac{\alpha}{n}-\frac1t}
\frac{
E_\alpha\!\left((E_\alpha f)^{t^\prime-1}\right)(x)
}{
\int_{\R^n}(E_\alpha f)^{t^\prime}\,dx
}
\right\}
\ \ 
\text{for a.e. }x\in\R^n,
\]
where \(A>0\) is chosen so that
\[
\int_{\R^n}f(x)\,dx=1.
\]
In particular, it is easy to see that $f$ is bounded.

\end{proof}

\begin{rem}

Using the lower bound \eqref{es-lower-bound-Ef-Rn} together with the Carlson--Levin inequality \eqref{LCLI}, one can also derive the rough inequality \eqref{ineq-main-Rn-log}, where
\[
\mathscr{C}(n,\alpha,1,t)
\ge
\mathscr C_0 \left[
\frac{\sigma(n)}{\alpha - n}\, 
\bigl(\mathscr{L}(n,\alpha-n,1)\bigr)^{
t^\prime + \frac{n}{\alpha - n}}  \,
B\!\left(\frac{n}{\alpha - n},\, -t^\prime - \frac{n}{\alpha - n}\right)\, 
\right]^{\frac{1}{t^\prime}}
>0.
\]
Moreover, if $\alpha=n+2$ and $\frac{n}{\alpha}<t<1$, then equality holds in \eqref{ineq-main-Rn-log} if and only if
\[
f(x)=\pi^{-\frac{n}{2}}e^{-|x|^2}
\]
up to translations and dilations. 
The corresponding sharp constant is given by
\[
\mathscr{C}(n,n+2,1,t)
=
\frac{1}{\pi}
\left(\frac{n}{2e}\right)^{1+\frac{n}{2t^\prime}}
\left(
\frac{\Gamma\!\left(-t^\prime-\frac{n}{2}\right)}{\Gamma(-t^\prime)}
\right)^{\frac{1}{t^\prime}}
.
\]
\end{rem}

\section{Proofs of theorems on \texorpdfstring{$\mathbb{R}^n_+$}{Rn+}} \label{Sec-half-space}

This section is devoted to the half-space case. We begin with the following rearrangement inequality on the upper half-space \cite[Lemma 2]{2017-NgoNgu-RH}. 

\begin{lem}[\cite{2017-NgoNgu-RH}]\label{lm-rearrangement-half-space}
Assume that $2\le n<\alpha$.
For any nonnegative function $f\in L^p(\partial\mathbb{R}^n_+)$, one has
\begin{equation}\label{ineq-rearrangement-half-space}
\|\widetilde E_\alpha f^*\|_{L^{t^\prime}(\mathbb{R}^n_+)}
\le
\|\widetilde E_\alpha f\|_{L^{t^\prime}(\mathbb{R}^n_+)},
\end{equation}
with equality if and only if  $f(x)=f^*(x+x_0)$ for some $x_0\in\partial\mathbb{R}^n_+$.
\end{lem}

\subsection{Rough reverse HLS inequality on \texorpdfstring{$\mathbb{R}^n_+$}{Rn+}}
Define
\[
\widetilde F[f]
    :=\frac{\|\widetilde E_\alpha f\|_{L^{t^\prime}(\R^n_+)}}{\|f\|^{\widetilde\gamma}_{L^1(\partial\R^n_+)}\|f\|^{1-\widetilde\gamma}_{L^q(\partial\R^n_+)}}.
\]
Consider the minimization problem on the upper half-space 
\begin{equation}\label{mini-problem-half-space} 
\widetilde{\mathscr{C}}(n,\alpha,q,t)
:=\inf\left\{ 
    \widetilde F[f]
:0\le f\in L^1(\partial\R^n_+) \cap L^q(\partial\R^n_+), f \not\equiv 0 \right\}.
\end{equation}
\medskip
As in the whole-space case, we first establish a necessary and sufficient condition for the inequality~\eqref{Ineq-Main-half-space}.

\begin{prop}\label{Prop-SNC-Rn_+}  
Let $2 \le n < \alpha$ and $\frac{n}{\alpha}<t<1$.
\begin{enumerate}
\item If $0 < q \le \frac{n-1}{\alpha-1}$, then $\widetilde{\mathscr{C}}(n,\alpha,q,t)=0$.
\item If $\frac{n-1}{\alpha-1}<q<1$, then $\widetilde{\mathscr{C}}(n,\alpha,q,t) > 0$.
\end{enumerate}
\end{prop}

\begin{proof}

To prove \textit{(i)}, we again divide the argument into two cases. 

\textbf{Case~1.} Assume that $0 < q < \frac{n-1}{\alpha-1} < \frac{n}{\alpha} < t < 1$.

Let $h, g \in C_c^{\infty}(\partial\mathbb{R}^n_+)$ be nonnegative functions and
\[
\int_{\partial\R^n_+} g(y)\,dy = 1.
\]
For $\varepsilon>0$ sufficiently small, define
\[
h_\varepsilon(x')
:=
h(x')
+
\varepsilon^{-\alpha+1}
g\!\left(\frac{x'}{\varepsilon}\right),
\ \ 
x'\in\partial\mathbb R^n_+,
\]
and
\[
H_\varepsilon(x',x_n)
:=
h_\varepsilon(x')
\mathbf 1_{[0,\varepsilon]}(x_n),
\ \ 
(x',x_n)\in\mathbb R^n_+.
\]
Define
\[
\widetilde I_\alpha[h_\varepsilon]
:=
\int_{\mathbb R^n_+}
H_\varepsilon(x)
\widetilde E_\alpha h_\varepsilon(x)
\,dx.
\]
By H\"older's inequality,
\[
\begin{aligned}
\widetilde{\mathscr C}(n,\alpha,q,t)
\le \widetilde F[h_\varepsilon]
&=
\frac{
\|\widetilde E_\alpha h_\varepsilon\|_{L^{t^\prime}(\mathbb R^n_+)}
}{
\|h_\varepsilon\|_{L^1(\partial\mathbb R^n_+)}^{\widetilde\gamma}
\|h_\varepsilon\|_{L^q(\partial\mathbb R^n_+)}^{1-\widetilde\gamma}
}
\\
&\le
\frac{
\widetilde I_\alpha[h_\varepsilon]
}{
\|h_\varepsilon\|_{L^1(\partial\mathbb R^n_+)}^{\widetilde\gamma}
\|h_\varepsilon\|_{L^q(\partial\mathbb R^n_+)}^{1-\widetilde\gamma}
\|H_\varepsilon\|_{L^t(\mathbb R^n_+)}
}
=: \widetilde G[h_\varepsilon].
\end{aligned}
\]
Therefore, it suffices to show that 
\[
\lim_{\varepsilon \to 0} \widetilde G[h_\varepsilon] = 0.
\]

A direct computation yields
\begin{equation}\label{estimate-NC-half-space-1}
\|h_\varepsilon\|_{L^1(\partial\mathbb R^n_+)}
=
\|h\|_{L^1(\partial\mathbb R^n_+)}
+
\varepsilon^{n-\alpha}.
\end{equation}
Moreover,
\begin{equation}\label{estimate-NC-half-space-2}
\|h_\varepsilon\|_{L^q(\partial\mathbb R^n_+)}^q
\ge
\widetilde C_1
\left(
\|h\|_{L^q(\partial\mathbb R^n_+)}^q
+
\varepsilon^{\,n-1-q(\alpha-1)}
\right).
\end{equation}
Similarly,
\begin{equation}\label{estimate-NC-half-space-3}
\|H_\varepsilon\|_{L^t(\R^n_+)}^t
=
\varepsilon
\|h_\varepsilon\|_{L^t(\partial\R^n_+)}^t
\ge
\widetilde C_2
\left(
\varepsilon\|h\|_{L^t(\partial\R^n_+)}^t
+
\varepsilon^{\,n-t(\alpha-1)}
\right).
\end{equation}
Furthermore,
\begin{align}\label{estimate-NC-half-space-4}
\widetilde I_\alpha[h_\varepsilon]
&=
\int_0^\varepsilon
\int_{\partial\mathbb R^n_+}
\int_{\partial\mathbb R^n_+}
\Bigl(
h(x')
+
\varepsilon^{-\alpha+1}
\sigma\!\left(\frac{x'}{\varepsilon}\right)
\Bigr)
\nonumber\\
&\ \ \ \ \cdot
\bigl(
|x'-y|^2+x_n^2
\bigr)^{\frac{\alpha-n}{2}}
\Bigl(
h(y)
+
\varepsilon^{-\alpha+1}
\sigma\!\left(\frac{y}{\varepsilon}\right)
\Bigr)
\,dy\,dx'\,dx_n
\nonumber\\
&=
\varepsilon
\left(
\int_{\partial\mathbb R^n_+}
\int_{\partial\mathbb R^n_+}
h(x')
|x'-y|^{\alpha-n}
h(y)
\,dy\,dx'
+
o_\varepsilon(1)
\right)
\nonumber\\
&\quad
+
2\varepsilon^{\,n-\alpha+1}
\left(
\int_{\partial\mathbb R^n_+}
h(x')
|x'|^{\alpha-n}
\,dx'
+
o_\varepsilon(1)
\right)
\nonumber\\
&\quad
+
\widetilde C_3
\varepsilon^{\,n-\alpha+1}.
\end{align}
Combining \eqref{estimate-NC-half-space-1}--\eqref{estimate-NC-half-space-4}, we obtain
\[
\scalebox{0.90}{$\displaystyle
\widetilde G[h_\varepsilon]
\le
\frac{
O(\varepsilon)
+
\varepsilon^{\,n-\alpha+1}
\bigl(
\widetilde C_4+O_\varepsilon(1)
\bigr)
}{
\left(
\|h\|_{L^1(\partial\mathbb R^n_+)}
+\varepsilon^{\,n-\alpha}
\right)^{\widetilde\gamma}
\left(
\widetilde C_1
\left(
\|h\|_{L^q(\partial\mathbb R^n_+)}^q
+
\varepsilon^{\,n-1-q(\alpha-1)}
\right)
\right)^{\frac{1-\widetilde\gamma}{q}}
\left(
\widetilde C_2
\left(
\varepsilon\|h\|_{L^t(\partial\mathbb R^n_+)}^t
+
\varepsilon^{\,n-t(\alpha-1)}
\right)
\right)^{1/t}
}.
$}
\]
Since $n-\alpha+1<1$, $n-1-q(\alpha-1)>0$, and $n-t(\alpha-1)<1$,
the preceding estimate reduces to 
\begin{align*}
\widetilde G[h_\varepsilon]
&\le
\frac{
\widetilde C_4
\varepsilon^{\,n-\alpha+1}
}{
\left(\widetilde C_1\|h\|_{L^q(\partial\R^n_+)}^q
\right)^{\frac{1-\widetilde\gamma}{q}}
\widetilde C_2^{\,1/t}
\,
\varepsilon^{\,\widetilde\gamma(n-\alpha)+\frac{n}{t}-(\alpha-1)}
}
+
o_\varepsilon(1).
\end{align*}
Therefore, it suffices to verify that
\[
n-\alpha +1 > \widetilde \gamma(n-\alpha)+\frac{n}{t}-(\alpha-1),
\]
or equivalently, 
\[
\widetilde \gamma > \frac{n}{t^\prime(n-\alpha)}.
\]
This is immediate since \(\widetilde\gamma(q)\) is strictly decreasing in \(q\in(0,1)\), and
\[
\widetilde\gamma\!\left(
\frac{n-1}{\alpha-1}
\right)
=
\frac{n}{t^\prime(n-\alpha)}.
\]
Hence we conclude that $\widetilde{\mathscr{C}}(n,\alpha,q,t)=0$ in the case $0<q<\frac{n-1}{\alpha-1}$.

\textbf{Case 2.}  Assume that $q=\frac{n-1}{\alpha-1}<\frac{n}{\alpha}<t<1$.
Define
\[
h_k(x)
:=
\begin{cases}
e^{-\alpha+1},
&
|x|\le e,
\\[0.3em]
|x|^{-\alpha+1}
(\log|x|)^{-\frac{\alpha-1}{n-1}\left(1+\frac1k\right)},
&
|x|>e.
\end{cases}
\]
Repeating the argument in Case~2 of the proof of Proposition~\ref{Prop-SNC-Rn}, one verifies that 
$$
\lim\limits_{k\to\infty}\widetilde F[h_k]
=\lim\limits_{k\to\infty}
\frac{\|\widetilde E_\alpha h_k\|_{L^{t^\prime}(\R^n_+)}}
{\|h_k\|_{L^1(\partial\R^n_+)}^{\widetilde \gamma}\,\|h_k\|_{L^q(\partial\R^n_+)}^{1-\widetilde\gamma}}
=0.
$$
Therefore, $\widetilde{\mathscr C}(n,\alpha,q,t)=0$ in this case.

\medskip
To prove \textit{(ii)}, it suffices by Lemma~\ref{lm-rearrangement-half-space} to consider radially symmetric non-increasing functions.

Arguing as in the proof of Proposition~\ref{Prop-SNC-Rn}, one obtains the lower bound
\begin{equation}\label{Estimate-LB-half-space}
\widetilde E_\alpha f(x)
\ge
\widetilde{\mathscr C}_0
\left(
\|f\|_{L^1(\partial\R^n_+)} |x|^{\alpha-n}
+
\int_{\partial\mathbb R^n_+}
|y|^{\alpha-n}f(y)\,dy
\right),
\end{equation}
for all $x\in\mathbb R^n_+$ and $0\le f=f^*\in L^1(\partial\mathbb R^n_+)$, where \(\widetilde{\mathscr C}_0>0\) is independent of \(f\) and \(x\).
Consequently,
\begin{align*}
\|\widetilde E_\alpha f\|_{L^{t^\prime}(\R^n_+)}
&\ge
\widetilde{\mathscr C}_0
\left(
\int_{\mathbb R^n_+}
\left(
\|f\|_{L^1(\partial\R^n_+)} |x|^{\alpha-n}
+
\int_{\partial\mathbb R^n_+}
|y|^{\alpha-n}f(y)\,dy
\right)^{t^\prime}
dx
\right)^{1/t^\prime}.
\end{align*}
Applying the Carlson--Levin inequality \eqref{CLI}, and 
using \eqref{Beta Type Function}, it follows that
\begin{align*}
\|\widetilde E_\alpha f\|_{L^{t^\prime}(\R^n_+)}
&\ge
\widetilde {\mathscr C}'
\,
\|f\|_{L^1(\partial\R^n_+)}^{
1-
\frac{q(\alpha-n)}{(n-1)(1-q)}
\left(
1+\frac{n}{t^\prime(\alpha-n)}
\right)
}
\cdot
\|f\|_{L^q(\partial\R^n_+)}^{
\frac{q(\alpha-n)}{(n-1)(1-q)}
\left(
1+\frac{n}{t^\prime(\alpha-n)}
\right)
},
\end{align*}
where
\begin{align*}
\widetilde {\mathscr C}'
:=
\widetilde{\mathscr C}_0
\Biggl[
\frac{\sigma(n)}{2(\alpha-n)}\,
\mathscr L(n-1,\alpha-n,q)^{
t^\prime+\frac{n}{\alpha-n}
}\,
B\!\left(
\frac{n}{\alpha-n},
-t^\prime-\frac{n}{\alpha-n}
\right)
\Biggr]^{1/t^\prime}.
\end{align*}
Since
\[
1-\widetilde\gamma
=
\frac{q(\alpha-n)}{(n-1)(1-q)}
\left(
1+\frac{n}{t^\prime(\alpha-n)}
\right),
\]
we conclude that
\[
\widetilde{\mathscr C}(n,\alpha,q,t)
\ge
\widetilde {\mathscr C}'
>
0.
\]
\end{proof}

\begin{coro} \label{coro-sharp-classification-half-space}
Let $\alpha=n+2$, $\frac{n}{n+2}<t<1$ and $\frac{n-1}{n+1}<q<1$. 
Then the extremal functions for \eqref{Ineq-Main-half-space} are given by
\[
f(x) = \bigl(1+|x|^2\bigr)^{-\frac{1}{1-q}},
\]
up to translations, dilations, and constant multiples.  
Moreover, the sharp constant is
\begin{align}
\widetilde{\mathscr C}(n,n+2,q,t)
={}&
\frac{1}{\pi}
\Biggl[
\left(\frac{q(n+1)-n+1}{2q}\right)^{
\frac{n+2t^\prime}{(n-1)(1-q)}
}
\left(\frac{(n-1)(1-q)}{q(n+1)-n+1}\right)^{
\frac n2+t^\prime
}
\nonumber\\
&\qquad\times
\left(
\frac{
\Gamma\!\left(\frac{1}{1-q}\right)
}{
\Gamma\!\left(\frac{1}{1-q}-\frac{n-1}{2}\right)
}
\right)^{\frac{n+2t^\prime}{n-1}}
\frac{
\Gamma\!\left(-t^\prime-\frac n2\right)
}{
2\Gamma(-t^\prime)
}
\Biggr]^{\frac1{t^\prime}} .
\label{const-sharp-half-space}
\end{align}




\end{coro}

\begin{proof}
We omit it here since the argument is almost the same as that in Corollary~\ref{coro-extremal-functions} . 
\end{proof}

\subsection{Existence of extremal functions for reverse HLS inequality on \texorpdfstring{$\mathbb{R}^n_+$}{Rn+}}
We next investigate the existence of nonnegative minimizers for \(\widetilde{\mathscr C}(n,\alpha,q,t)\).

\begin{prop}\label{prop-Minimizer-half-space}   
Let $2\le n<\alpha $, $\frac{n}{\alpha}<t<1$, and $\frac{t^\prime(n-1)}{n+t^\prime(\alpha-1)}\le q<1$. 
Then the infimum \(\widetilde{\mathscr C}(n,\alpha,q,t)\) is achieved. 
Moreover, any extremal function is radially symmetric (up to a translation on \(\partial\mathbb R^n_+\)), monotone non-increasing and positive almost everywhere on $\partial\mathbb R^n_+$.
\end{prop}

\begin{proof}
If $q=\frac{t^\prime(n-1)}{n+t^\prime(\alpha-1)}$, then $\gamma=0$. This is the case due to \cite{2017-NgoNgu-RH}.

In the following we always assume
\[
\frac{t^\prime(n-1)}{n+t^\prime(\alpha-1)}<q<1.
\]
Then \(\widetilde\gamma<0\).

We consider the minimization problem \eqref{mini-problem-half-space}
and let \(\{f_j\}\) be a minimizing sequence, namely,
\[
\lim_{j\to\infty}\widetilde F[f_j]
=
\widetilde{\mathscr C}(n,\alpha,q,t).
\]
For convenience, we divide the proof into three steps.

\medskip
\textbf{\textit{Step 1.}} Construction of a candidate minimizer. 

By Lemma~\ref{lemma-rearrangement}~and~Lemma \ref{lm-rearrangement-half-space}, we may assume that each \(f_j\) is nonnegative, radially symmetric, and non-increasing on \(\partial\mathbb R^n_+\).
Using the scaling invariance and homogeneity of \(\widetilde F\), we normalize the sequence so that
\[
\|f_j\|_{L^1(\partial\mathbb R^n_+)}
=
\|f_j\|_{L^q(\partial\mathbb R^n_+)}
=
1,
\ \ 
\forall j\in\mathbb N.
\]
Arguing exactly as in Proposition~\ref{prop-Minimizer}, we obtain
\begin{equation}\label{estimate-minimizing-sequence-half-space}
f_j(r)
\le
C\min\left\{
r^{-\frac{n-1}{q}},
\,r^{-n+1}
\right\},
\ \ 
\forall r>0,
\end{equation}
where \(C>0\) is independent of \(j\).
By Helly's selection theorem, after passing to a subsequence if necessary, there exists a measurable function \(f_0\) such that
\[
f_j\to f_0
\ \ \text{a.e. on }\partial\mathbb R^n_+.
\]
Moreover, $f_0$ is nonnegative, radially symmetric, and non-increasing on \(\partial\mathbb R^n_+\).

\medskip
\textbf{\textit{Step 2.}} Proof of the inequality

\[
\|\widetilde E_\alpha f_0\|_{L^{t^\prime}(\mathbb R^n_+)}
\le
\lim_{j\to\infty}
\|\widetilde E_\alpha f_j\|_{L^{t^\prime}(\mathbb R^n_+)}.
\]
It can be obtained by an argument as that in Step~2 of the proof of Proposition~\ref{prop-Minimizer}.

\medskip
\textbf{\textit{Step 3.}} Attainment of the infimum.


As that in Proposition~\ref{prop-Minimizer}, for any $s\in\left(\frac{n-1}{\alpha-1},q\right)$,  Carlson--Levin inequality \eqref{CLI} implies that $f_j\in L^s(\partial\mathbb{R}^n_+)$.
Then applying the rough inequality \eqref{Ineq-Main-half-space}, we obtain
\[
\|\widetilde E_\alpha f_j\|_{L^{t^\prime}(\R^n_+)}
\ge
\widetilde{\mathscr{C'}}\|f_j\|_{L^s(\partial\R^n_+)}^{\,1-\widetilde\gamma(s)},
\]
where $1-\widetilde\gamma(s)>0$. 
Consequently, $\{f_j\}$ is uniformly bounded in $L^s(\partial\mathbb{R}_+^n)$. Then
\[
f_j(|x|)
\le
C|x|^{-\frac{n-1}{s}},
\ \ 
\forall j\in\mathbb N.
\]
Combining this estimate with \eqref{estimate-minimizing-sequence-half-space}, we deduce that
\[
f_j(|x|)
\le
C\min\left\{
|x|^{-\frac{n-1}{s}},
\,|x|^{-n+1}
\right\}
\in
L^q(\partial\mathbb R^n_+).
\]
Therefore, the dominated convergence theorem yields
\[
\|f_0\|_{L^q(\partial\mathbb R^n_+)}
=
\lim_{j\to\infty}
\|f_j\|_{L^q(\partial\mathbb R^n_+)}
=
1.
\]
The properties of extremal functions follow as in the proof of Proposition~\ref{prop-positivity-Rn}.

\end{proof}

\begin{rem} \label{rem-semi-continuous-half-space}
The map $q\mapsto \widetilde{\mathscr{C}}(n,\alpha,q,t)$ is upper semi-continuous.
In particular,
\[
\lim_{q\to\left(\frac{n-1}{\alpha-1}\right)^+}
\widetilde{\mathscr C}(n,\alpha,q,t)
=
0.
\]
\end{rem}

\subsection{Inequalities for \texorpdfstring{$q\ge1$}{q>1} on \texorpdfstring{$\mathbb{R}^n_+$}{Rn+}}
The proofs of Theorems~\ref{them-half-space-q-bigger-1} and \ref{them-half-space-q=1-log} are analogous to those of Theorems~\ref{them-Rn-q-bigger-1} and \ref{them-Rn-q=1-log} respectively.
We therefore omit the details.

Similarly, the corresponding inequalities \eqref{ineq-main-half-space-q-bigger-1} and \eqref{ineq-main-half-space-log} can also be derived from the Carlson--Levin inequality.
In particular, when $\alpha=n+2$, the sharp constant can be computed explicitly.

 \vskip 1cm
\noindent {\bf Acknowledgements}\\
\noindent

The project is supported by  the
National Natural Science Foundation of China (Grant Nos. 12371119, 12271436).




\bibliographystyle{elsarticle-num-Z2}
\bibliography{Zcite_260808}

\end{document}